\documentclass{amsart}
\usepackage[utf8]{inputenc}
\usepackage{xy}
\xyoption{all}
\usepackage{amssymb}
\usepackage{amsmath}
\usepackage{amsthm}
\usepackage{mathtools}
\usepackage{hyperref}
\usepackage{cleveref}
\usepackage{xcolor}
\usepackage{comment}
\usepackage{txfonts} 
\usepackage{tikz,tikz-cd}
\usepackage{url}

\newtheorem{theorem}{Theorem}[section]
\newtheorem{proposition}[theorem]{Proposition}
\newtheorem{definition-proposition}[theorem]{Definition-Proposition}
\newtheorem{lemma}[theorem]{Lemma}
\newtheorem{corollary}[theorem]{Corollary}
\newtheorem{definition}[theorem]{Definition}

\newcounter{lettered}
\newtheorem{letteredtheorem}[lettered]{Theorem}

\theoremstyle{definition}
\newtheorem{remark}[theorem]{Remark}

\DeclareMathOperator{\holim}{{\rm holim}}
\DeclareMathOperator{\colim}{{\rm colim}}
\DeclareMathOperator{\Spf}{{\rm Spf}}

\DeclareMathOperator{\Def}{{\rm Def}}
\DeclareMathOperator{\act}{\cdot}

\newcommand{\wt}{{\rm wt}}
\newcommand{\ind}{{\rm ind}}

\DeclareMathOperator{\im}{{\rm im\:}}

\newcommand{\J}{\Lambda}

\newcommand{\Gal}{{\rm Gal}}
\newcommand{\Aut}{{\rm Aut}}

\title[The action of the Morava stabilizer group at height $2$]{The action of the Morava stabilizer group on the coefficients of Morava $E$-theory at height $2$}
\author{Andrew Salch}
\date{January 2025}

\begin{document}

\maketitle

\begin{abstract}
It is an old problem in stable homotopy theory to calculate the action of the full Morava stabilizer group on the ring of homotopy groups of Morava $E$-theory. We solve this problem completely at height $2$ by calculating an explicit closed formula for that action. In particular, this yields an explicit, surprisingly simple closed formula for the action of the automorphism group of a height $2$ formal group law on its Lubin-Tate deformation ring. The formula is of a combinatorial nature, given by sums over certain labelled ordered rooted trees.
\end{abstract}

\tableofcontents

\section{Introduction}
It is an old problem, in stable homotopy theory and in arithmetic geometry, to give an explicit formula for the action of the automorphism group of a finite-height formal group law on its Lubin--Tate deformation ring. Most efforts have focused on the first nontrivial case, the case of height $2$. In this paper we solve this problem completely at height $2$ at all primes by calculating an explicit, elementary closed formula for this group action. 

Before stating the precise result (contained in Theorems \ref{lettered thm A} and \ref{lettered thm B}, below), we will explain what any of this means, why it matters, and the history of the problem.

\subsection{History of the problem}
\label{History of the problem}

Let $\mathbb{G}$ be a one-dimensional formal group law of height $h$ over the finite field $\mathbb{F}_{p^h}$. The Lubin--Tate ring $\Def(\mathbb{G})$, from \cite{MR0238854}, classifies deformations\footnote{Readers seeking a precise statement of what the word ``deformations'' means in this context can consult \cref{Review of Devinatz...}, below. That section also contains a review of other relevant ideas from the deformation theory of formal groups, including some which are known but not often stated explicitly.} of $\mathbb{G}$ to Artinian local commutative rings with residue field $\mathbb{F}_{p^h}$. The ring $\Def(\mathbb{G})$ is isomorphic to the power series algebra $W(\mathbb{F}_{p^h})[[u_1,u_2, \dots ,u_{h-1}]]$ over the Witt ring $W(\mathbb{F}_{p^h})$. To be perfectly explicit: the Witt ring $W(\mathbb{F}_{p^h})$ isomorphic to the ring $\hat{\mathbb{Z}}_p(\zeta_{p^h-1})$ of $p$-adic integers with a primitive $(p^h-1)$th root of unity adjoined.

In stable homotopy theory, the automorphism group $\Aut(\mathbb{G})$ of the formal group law $\mathbb{G}$ is called the {\em full Morava stabilizer group of height $h$}, after Morava's insights about the applications of $\Aut(\mathbb{G})$ to computational problems in stable homotopy theory in \cite{MR782555} (see \cite[chapter 6]{MR860042} for a textbook treatment). We sketch some of those applications below, in \cref{applications for...}. 

The Morava stabilizer group $\Aut(\mathbb{G})$ acts on $\Def(\mathbb{G})$ in a natural way. The action is trivial in the case $h=1$. However, it is notoriously difficult to describe explicitly for any height $h$ greater than $1$. The group $\Aut(\mathbb{G})$ itself is not so complicated: it is a profinite group, isomorphic to the group of units in the maximal order in the invariant $1/n$ central division algebra over $\mathbb{Q}_p$, and consequently it admits the explicit presentation
\begin{align}
\label{presentation 1} \Aut(\mathbb{G}) &\cong \left( W(\mathbb{F}_{p^h})\langle S \rangle/\left(S^h = p, \omega^{\sigma} S = S\omega \mbox{\ for\ each\ }\omega\in W(\mathbb{F}_{p^h})\right)\right)^{\times}, 
\end{align}
where the function $x\mapsto x^{\sigma}$ is the canonical lift of the Frobenius automorphism of $\mathbb{F}_{p^h}$ to an automorphism of the ring $W(\mathbb{F}_{p^h})$. 

Let us now restrict to the first open case, i.e., the case of height $h=2$. To fix ideas, let $\mathbb{G}$ be the height $2$ Honda formal group law over $\mathbb{F}_{p^2}$; see \cref{Review of Devinatz...} for a precise definition. The Lubin--Tate ring $\Def(\mathbb{G})$ is simply a power series algebra $W(\mathbb{F}_{p^2})[[u_1]]$ on a single generator, $u_1$, the Hazewinkel generator. Using the presentation \eqref{presentation 1}, each element of $\Aut(\mathbb{G})$ can be written uniquely as a $W(\mathbb{F}_{p^2})$-linear combination $\alpha_0 + \alpha_1S$, with $\alpha_0$ a unit. Since $\Aut(\mathbb{G})$ acts on $\Def(\mathbb{G})$ continuously and by $W(\mathbb{F}_{p^2})$-algebra homomorphisms, the question of calculating the action of $\Aut(\mathbb{G})$ on $\Def(\mathbb{G})$ is simply the question of calculating the power series 
\begin{align*}
 (\alpha_0 + \alpha_1S).u_1 &\in W(\mathbb{F}_{p^2})[[u_1]],
\end{align*}
as a function of $\alpha_0$ and $\alpha_1$ (and $p$).

For concreteness, here is an explicit example. 
Let $p=2$, and let $\zeta_3$ denote a primitive cube root of unity in $W(\mathbb{F}_4)$, so that $W(\mathbb{F}_4) = \hat{\mathbb{Z}}_2(\zeta_3)$. Then $1+p \zeta_3 \in W(\mathbb{F}_4)^{\times}\subseteq \Aut(\mathbb{G})$ acts on $u_1$ as follows\footnote{The formulas \eqref{p=2 example} and \eqref{p=3 example} were not calculated using the new theorems in this paper. Rather, we calculated them by writing a software implementation of ideas of Devinatz--Hopkins \cite{MR1333942} for making {\em approximate} calculations of the action of $\Aut(\mathbb{G})$ on $\Def(\mathbb{G})$. The ideas of Devinatz--Hopkins are recalled below in \cref{Review of Devinatz...}. Our software implementation is available on \href{http://asalch.wayne.edu/}{the author's webpage}.}:
\begin{align}\label{p=2 example}
 (1 + \zeta_3 p).u_1 &= -u_1 - u_1^4 - 2 u_1^7 - 4 u_1^{10} - 7 u_1^{13} + 8 u_1^{16} + 168 u_1^{19} + 1099 u_1^{22}  \\
\nonumber &\hspace{10pt} + 5356 u_1^{25} + 20958 u_1^{28} + 60540 u_1^{31} + 57790 u_1^{34} - 820040 u_1^{37}  \\
\nonumber &\hspace{10pt} - 8153100 u_1^{40} - 50769012 u_1^{43} - 248695951 u_1^{46}   \\ 
\nonumber &\hspace{10pt} - 972324035 u_1^{49}- 2637054605 u_1^{52} - 171581978 u_1^{55}   \\
\nonumber &\hspace{10pt} + 60417121849 u_1^{58} + 530425333282 u_1^{61} + 3196481834506 u_1^{64}  \\
\nonumber &\hspace{10pt} + 15298429872978 u_1^{67} + 57330724580351 u_1^{70} \mod u_1^{73}.
\end{align}
The reader who enjoys puzzles is welcome to try to spot a pattern\footnote{The example \eqref{p=2 example} is indeed probably the simplest nontrivial example of the action of an element of $\Aut(\mathbb{G})$ on $u_1$. It has the disadvantage, though, that perhaps it is too close to being trivial: it happens to coincide with the $2$-series of the height $2$ Honda formal group law over $\mathbb{F}_2$. In that sense it is perhaps not representative of the action of a ``randomly-chosen'' element of $\Aut(\mathbb{G})$ on $u_1$. For a somewhat more representative example, let $p=3$. Then $1+ p\sqrt{-1} \in W(\mathbb{F}_9)^{\times}\subseteq \Aut(\mathbb{G})$ acts on $u_1$ as follows:
\begin{align}\nonumber
 \left(1+ p\sqrt{-1}\right).u_1
  &= \frac{1}{5}\left(-3\sqrt{-1} - 4\right)u_1 + \frac{2^4}{5^5}\left(44\sqrt{-1} + 117\right)u_1^5 + \frac{2^4}{5^8}\left(-23708\sqrt{-1} - 23269\right)u_1^9 \\
\nonumber &\hspace{10pt}  + \frac{2^5}{5^{13}}\left(165420188\sqrt{-1} - 42317341\right)u_1^{13} + \frac{2^4}{5^{17}}\left(-1177652602636\sqrt{-1} + 573119181627\right)u_1^{17}\\
\label{p=3 example} &\hspace{10pt} + \frac{2^5}{5^{21}}\left(1430942855147644\sqrt{-1} - 1030288432269333\right)u_1^{21}\\
\nonumber &\hspace{10pt} + \frac{2^7}{5^{25}}\left(-483182830183187406\sqrt{-1} + 1115554534647772267\right)u_1^{25} \\
\nonumber &\hspace{10pt}+ \frac{2^4}{5^{28}}\left(-590648668256168586804\sqrt{-1} - 7472490792127549191047\right)u_1^{29} \mod u_1^{33}.
\end{align}
} that describes the coefficients in \eqref{p=2 example} before glancing at Theorem \ref{lettered thm A}, where the answer is given, and further discussed at the end of \cref{Drawn examples...}.

Here is a survey of the existing calculations of the action of $\Aut(\mathbb{G})$ on $\Def(\mathbb{G})$. All the existing calculations are approximations, i.e., calculations of $(\alpha_0 + \alpha_1S).u_1$ modulo $p$ and some power of $u_1$. 
\begin{itemize}
\item For all heights $h$ and all primes $p$, Chai's 1996 paper \cite[Theorem 1]{MR1387691} calculates $\alpha.u_i$ modulo $(p,u_1^2,u_2^2,\dots ,u_{h-1}^2)$ for all $i$. (Chai gets better accuracy in the case where $\mathbb{G}$ is not only a formal group law, but a formal $A$-module, where $A$ is a complete discrete valuation ring whose residue field is $\mathbb{F}_{p^r}$, specifically in the case $r>1$. In that case, at height $h=2$, Chai gets $(\alpha_0 + \alpha_1S).u_i$ modulo $(p,u_1^{p^r})$. Chai only gets recursive formulas, however, not a closed formula.)
\item In Karamanov's 2006 Ph.D. thesis \cite[section 4.2]{karamanovthesis}, Karamanov calculated $(\alpha_0 + \alpha_1S).u_1$ modulo $(p,u_1^7)$ in the case $p=3$. Slightly higher accuracy is achieved, namely calculations modulo $(p,u_1^{10})$, for certain specific elements $\alpha_0 + \alpha_1S\in \Aut(\mathbb{G})$. These results were published in Karamanov's 2013 paper \cite[section 4.3]{MR3127044} with Henn and Mahowald.
\item Lader's 2013 Ph.D. thesis \cite[section 3.2]{laderthesis} calculates $(\alpha_0 + \alpha_1S).u_1$ modulo $(p,u_1^{2p+1})$ and modulo $(p,u_1^{p+4})$ at all primes.
\item Kohlhaase's 2013 paper \cite[Theorem 1.16]{MR3069363} calculated $(\alpha_0+\alpha_1S).u_1$ modulo $(p,u_1^{2p+2})$, but only ``for a limited number of elements'' $\alpha_0+\alpha_1S\in \Aut(\mathbb{G})$. Kohlhaase also made approximate calculations of the action of $\Aut(\mathbb{G})$ on $\Def(\mathbb{G})$ at arbitrary height, but to lower accuracy, specifically modulo $p+\mathfrak{m}^{p+2}$ where $\mathfrak{m}$ is the maximal ideal in $\Def(\mathbb{G})$. Finally, Kohlhaase made a higher-accuracy calculation \cite[Theorem 1.19]{MR3069363} of the action of the subgroup $W(\mathbb{F}_{p^2})^{\times}$ of $\Aut(\mathbb{G})$ on $\Def(\mathbb{G})$: at height $h=2$ he gets the action modulo $(p,u_1^{2p+4})$ for all primes $p>3$.
\item For all heights $h>2$, Davis's 2022 Ph.D. thesis \cite[Theorem 4.16]{MR4435205} calculates the power series $\alpha.u_{h-1}$ modulo $(p,u_1,u_2, \dots ,u_{h-2},u_{h-1}^{p^{h-1}+1})$ at all primes $p$.
\end{itemize}

\subsection{Motivations for the problem}
\label{applications for...}

There are good reasons why both stable homotopy theorists and $p$-adic geometers have made calculations of the action of $\Aut(\mathbb{G})$ on $\Def(\mathbb{G})$. We describe the stable-homotopical applications first. The formal group law $\mathbb{G}$ has an associated {\em Morava $E$-theory spectrum} $E(\mathbb{G})$, a spectrum in the sense of stable homotopy theory, constructed by Morava in \cite{MR782555}. Morava showed that the degree zero subring $E(G)_0 = \pi_0(E(\mathbb{G}))$ is isomorphic to the Lubin--Tate deformation ring $\Def(\mathbb{G})$, and furthermore that the graded ring $E(\mathbb{G})_* = \pi_*(E(\mathbb{G}))$ is isomorphic to $\Def(\mathbb{G})[u^{\pm 1}]$, that is, $\Def(\mathbb{G})$ (concentrated in degree zero) with a Laurent polynomial generator $u$ adjoined in degree $-2$. The graded ring $\Def(\mathbb{G})[u^{\pm 1}]$ has a description as the classifying ring of deformations of $\mathbb{G}$ marked with a unit; see \cref{Review of Devinatz...} for a precise formulation. The automorphism group $\Aut(\mathbb{G})$ acts naturally on this moduli problem of deformations, hence acts naturally on $E(\mathbb{G})_*$.

Goerss and Hopkins proved \cite{MR2125040} that the Morava stabilizer group $\Aut(\mathbb{G})$ acts on the {\em spectrum} $E(\mathbb{G})$ by $E_\infty$-ring spectrum automorphisms, such that upon taking homotopy groups, the resulting action of $\Aut(\mathbb{G})$ on $E(\mathbb{G})_*$ is the natural action described in the previous paragraph. 
Given a finite CW-complex $X$, there exist spectral sequences \cite{MR2030586}
\begin{align}
\label{dhss 1} E_2^{s,t}\cong H^s(\Aut(\mathbb{G}); E(\mathbb{G})_t(X)) & \Rightarrow \pi_{t-s}\left( (E(\mathbb{G})\wedge X)^{h\Aut(\mathbb{G})}\right)\mbox{\ \ and} \\
\label{dhss 2} E_2^{s,t}\cong H^s\left(\Gal(\mathbb{F}_{p^h}/\mathbb{F}_p); \pi_{t}\left( (E(\mathbb{G})\wedge X)^{h\Aut(\mathbb{G})}\right)\right) & \Rightarrow \pi_{t-s}\left( \pi_{K(h)}X\right),
\end{align}
where $\pi_*(L_{K(h)}X)$ is the Bousfield localization of $X$ at the height $h$ Morava $K$-theory. One can be precise about what kind of object $(E(\mathbb{G})\wedge X)^{h\Aut(\mathbb{G})}$ is, but for present purposes, it does not matter: one can simply think of $\pi_*((E(\mathbb{G})\wedge X)^{h\Aut(\mathbb{G})})$ as an algebraic waypoint between the cohomology $H^*(\Aut(\mathbb{G}); E(\mathbb{G})_*(X))$ and the $K(h)$-local homotopy groups $\pi_*(L_{K(h)}X)$. The second spectral sequence \eqref{dhss 2} is quite simple: note that, if $p\nmid h$, then the spectral sequence collapses to the $s=0$-line, i.e,, $\pi_{*}\left( \pi_{K(h)}X\right)$ is simply the $\Gal(\mathbb{F}_{p^h}/\mathbb{F}_p)$-fixed points in the output of \eqref{dhss 1}. Spectral sequence \eqref{dhss 1} also collapses at the $E_2$-page for sufficiently large primes, namely, for $p>\frac{h^2+h+2}{2}$. 

Hence, at large primes, $\pi_*(L_{K(h)}X)$ simply {\em is} the $\Gal(\mathbb{F}_{p^h}/\mathbb{F}_p)$-fixed points of\linebreak $H^*(\Aut(\mathbb{G});E(\mathbb{G})_*(X))$, up to extensions in the $E_{\infty}$-page of the spectral sequence. The homotopy groups $\pi_*(L_{K(h)}X)$ are used in ``chromatic fracture squares'' to calculate the homotopy groups of $X$ localized at Johnson--Wilson theories, $\pi_*(L_{E(h)}X)$, and the chromatic convergence theorem \cite[Theorem 7.5.7]{MR1192553} ensures that $\pi_*(\holim_{h\rightarrow\infty} L_{E(h)}X)$ agrees with the $p$-localization of the stable homotopy groups of the CW-complex $X$ itself. 

Our point is that one has a means of calculating the notoriously difficult stable homotopy groups of finite CW-complexes (e.g. the stable homotopy groups of spheres) by some process which begins with calculating $H^*(\Aut(\mathbb{G});E(\mathbb{G})_*(X))$. This has been carried out in practice through heights $h\geq 2$ and, in a small handful of cases, at heights $3$ and $4$. Thus far, it seems that the calculation of $H^*(\Aut(\mathbb{G});E(\mathbb{G})_*(X))$ is the hardest part of this process. 

In the base case $X=S^0$, spectral sequences \eqref{dhss 1} and \eqref{dhss 2} jointly take as input\linebreak $H^*(\Aut(\mathbb{G});\Def_*(\mathbb{G}))$, and produce as output the $K(h)$-local stable homotopy groups of spheres $\pi_*(L_{K(h)}S^0)$. At large primes, both spectral sequences collapse immediately. Our point is this: {\em to calculate stable homotopy groups of finite CW-complexes, most especially stable homotopy groups of spheres, one would really like to be able to calculate $H^*(\Aut(\mathbb{G});\Def_*(\mathbb{G}))$. }

However, we are talking about calculating the cohomology of a group action---the action of $\Aut(\mathbb{G})$ on $\Def_*(\mathbb{G})$---for which nobody has been able to write down an explicit formula. To date, calculations of $H^*(\Aut(\mathbb{G}); \Def_*(\mathbb{G}))$ have been made by filtering $\Def_*(\mathbb{G})$ in such a way that one gets a spectral sequence converging to $H^*(\Aut(\mathbb{G}); \Def_*(\mathbb{G}))$ and whose differentials can be calculated in terms of the action of $\Aut(\mathbb{G})$ on $\Def_*(\mathbb{G})$ in some appropriate filtration quotient. This is why the approximate calculations of the action of $\Aut(\mathbb{G})$ on $\Def_*(\mathbb{G})$, surveyed in \cref{History of the problem}, have been fruitful for stable homotopy theory.

We will say less about how the action of $\Aut(\mathbb{G})$ on $\Def(\mathbb{G})$ is used in arithmetic geometry, simply because the author's personal motivation comes more directly from topological problems than from arithmetic-geometric problems. One significant application in $p$-adic geometry is as follows: the Raynaud generic fibre of the formal scheme $\Spf\Def(\mathbb{G})$ is the base of a tower of rigid analytic spaces, the {\em Lubin--Tate tower}, which, in the limit of the tower, admits an action of the profinite group $\Aut(\mathbb{G})\times GL_h(\mathbb{Q}_p)$. Write $LT_j(\mathbb{G})$ for the $j$th rigid analytic space in this tower.
For any prime $\ell\neq p$, the induced action of $\Aut(\mathbb{G})\times GL_h(\mathbb{Q}_p)$ on the compactly-supported \'{e}tale cohomology of the tower, \[\underset{j\rightarrow\infty}{\colim}\ \left( \mathbb{Q}_{\ell}\otimes_{\hat{\mathbb{Z}}_{\ell}} \lim_{i\rightarrow\infty} H^{*}_c(LT_j(\mathbb{G});\mathbb{Z}/\ell^i\mathbb{Z})\right) , \] realizes the $\ell$-adic Jacquet--Langlands correspondence between appropriate irreducible representations of $GL_h(\mathbb{Q}_p)$ and appropriate irreducible representations of the unit group of the invariant $1/h$ central division algebra over $\mathbb{Q}_p$ (i.e., $D_{1/h,\mathbb{Q}_p}^{\times}$, as opposed to $\mathcal{O}_{D_{1/h,\mathbb{Q}_p}}^{\times}\cong \Aut(\mathbb{G})$). The author's understanding of the history is that this fact about the cohomology of $LT(\mathbb{G})$ was conjectured by Carayol \cite{MR1044827}, and proved by Harris and Taylor \cite{MR1876802} and by Henniart \cite{MR1738446}. Our attempt to tell some of this story for a stable-homotopical audience is available in the preprint \cite{salchandstrauch}. 

One would like to understand the action of $\Aut(\mathbb{G})$ on the levels in the Lubin--Tate tower. At height $2$, the action on the base of the tower---to be clear, not only the action {\em rationally} on the generic fibre of $\Spf\Def(\mathbb{G})$ (studied by Gross--Hopkins in \cite{MR1217353} and \cite{MR1263712}), where $p$ has been inverted, but instead the full {\em integral} action on the formal scheme $\Spf\Def(\mathbb{G})$---is the subject of this paper.

\subsection{The main results}
\label{The main results}

Now we state the main results. The reason for the unpredictable behavior of the coefficients in the power series $(\alpha_0 + \alpha_1S).u_1$, for example the coefficients in \eqref{p=2 example} and \eqref{p=3 example}, is that these coefficients turn out to be sums {\em over certain families of labelled trees}. Specifically:
\begin{letteredtheorem} \label{lettered thm A}
Let $\gamma_n\in W(\mathbb{F}_{p^2})$ be the coefficient of $u_1^n$ in the power series $(\alpha_0 + \alpha_1S).u_1$, i.e., 
\begin{align*}
 (\alpha_0 + \alpha_1S).u_1 &= \gamma_0 + \gamma_1 u_1 + \gamma_2 u_1^2 + \gamma_3 u_1^3 + \dots .
\end{align*}
Then $\gamma_n$ is the sum, over all $p$-labelled ordered rooted trees $T$ of weight $n$, of the $(\alpha_0,\alpha_1)$-index of $T$:
\begin{align}\label{formula 34094}
 \gamma_n &= \sum_{\substack{T\in pLT,\\ \wt(T) = n}} \ind_{\alpha_0,\alpha_1}(T).
\end{align}
\end{letteredtheorem}
Theorem \ref{lettered thm A} is a special case of Theorem \ref{main action thm}, below, which handles formal modules and not only formal group laws. (In this introduction, we only present our results as they apply to formal group laws, not the more general results on formal modules which appear in the body of the paper.) 

In \eqref{formula 34094}, $pLT$ is the set of $p$-labelled finite ordered rooted trees. These are elementary combinatorial objects, and we give their definition in a moment, along with the definition of the weight and $(\alpha_0,\alpha_1)$-index of such a tree. 
But first, we give some explanation of why it is worth the reader's trouble to bother with such definitions. Before this paper, trees had not been used in calculations or descriptions of the action of $\Aut(\mathbb{G})$ on $\Def(\mathbb{G})$, and it was not known or conjectured that there was any such relationship. The existing calculations surveyed in \cref{History of the problem} did not use trees, and they gave reasonably simple, understandable formulas for the action of $(\alpha_0 + \alpha_1S).u_1$ in low degrees. Those low degrees turn out to be the degrees $n$ in which there are very few $p$-labelled trees of weight $n$---see for example the proof of Corollary \ref{witt comp cor 1}. Those calculational efforts foundered upon reaching degrees $n$ in which the number of $p$-labelled trees of weight $n$ suddenly grows quickly. If one does not know that the number of summands is indexed by the number of appropriately-labelled trees, then this sudden change in the patterns in the coefficients of the power series $(\alpha_0 + \alpha_1S).u_1$ seems unpredictable and totally unmanageable (see \cref{p=2 example} to get a sense of this). Our point is that one has to introduce trees with labellings, or something equivalent to them, in order to understand the patterns in the coefficients in the action of $(\alpha_0 + \alpha_1S).u_1$.

Now we give the definitions from combinatorics which are required to make sense of the statement of Theorem \ref{lettered thm A}. We start with the most well-known definitions, and proceed to some new ones (e.g. weight and index, which are new definitions). We refer the reader to Definitions \ref{def of rooted tree}, \ref{def of q-labelling}, and \ref{def of index} for full details.
\begin{itemize}
\item A {\em tree} is an connected acyclic graph. 
\item A {\em rooted tree} is a tree with a distinguished choice of vertex. 
\item An {\em ordered rooted tree} is a rooted tree with, for each $n$, a total ordering on the set of vertices of distance $n$ from the root vertex, such that the ordering is preserved by the edges; see \cref{def of rooted tree} for the precise definition.
\item Given vertices $v,w$ in a rooted tree, we say that $w$ is a {\em descendant of $v$} if the path from the root to $w$ passes through $v$. In particular, $v$ is a descendant of $v$ itself.
\item We say that $w$ is a {\em immediate descendant of $v$} if $w$ is a descendant of $v$ and there is an edge from $v$ to $w$.
\item A finite-length sequence of nonnegative integers is {\em parity-alternating} if each of its entries, after the first, is of opposite parity from the previous entry. We write $\J$ for the set of parity-alternating monotone strictly increasing sequences of nonnegative integers which do not begin with an odd number.
\item Given a prime number $p$ and a finite sequence of nonnegative integers $I = (i_1, \dots,i_n)$, we write $QI$ for the sum $p^{i_1} + \dots + p^{i_n}$. We write $\left| I\right|$ for the length of $I$, i.e., $\left| (i_1,\dots ,i_n)\right|=n$.
\item Given a prime number $p$, a {\em $p$-labelling} of an ordered rooted tree $T$ is a choice, for each vertex $v$ of $T$, of an ordered pair $(H_v,I_v)$ of elements of $\J$, called the {\em label of $v$}, which are required to satisfy the following conditions:
\begin{itemize}
\item For each vertex $v$, the number of immediate descendants of $v$ is equal to $QH_v$.
\item The weight of each vertex is positive, and strictly greater than the weight of each of its immediate descendants. Here by the {\em weight} of a vertex, we mean the sum of $QI_w$ taken across all descendants $w$ of $v$:
\begin{align*}
 \wt(v) &= \sum_{\substack{\text{descendants}\\w\text{\ of\ } v}} QI_w.
\end{align*}
\end{itemize}
\item We define the weight of a $p$-labelled tree as the weight of its root vertex.
\item Finally, given a prime $p$ and elements $\alpha_0,\alpha_1$ of an integral domain equipped with an involution $\sigma$, the {\em $(\alpha_0,\alpha_1)$-index} of a vertex $v$ is 
\[ \frac{(-1)^{\left| H_v\right|} \sigma^{\left| I_v\right|}(\alpha_{\left| H_v\right| + \left| I_v\right| - 1})}{p^{\lfloor (\left| H_v\right| + \left| I_v\right|-1)/2\rfloor} \alpha_0},\] where the subscript of $\alpha$ is understood as being defined modulo $2$, and the symbol $\lfloor n\rfloor$ denotes the integer floor of $n$. The $(\alpha_0,\alpha_1)$-index of a $p$-labelled tree $T$ is defined as the product of the $(\alpha_0,\alpha_1)$-indices of all the vertices in $T$.

\end{itemize}
When these notions are applied to the action of $\Aut(\mathbb{G})$ on the Lubin--Tate ring, the involution $\sigma$ is the Galois conjugation on $W(\mathbb{F}_{p^2})$. For notational convenience we sometimes write $\alpha \mapsto \overline{\alpha}$ rather than $\alpha \mapsto \sigma(\alpha)$. Here is an example of a $2$-labelled tree of weight $6$ and $(\alpha_0,\alpha_1)$-index $\frac{-\overline{\alpha}_0^5\alpha_1^2\overline{\alpha}_1}{\alpha_0^8}$:
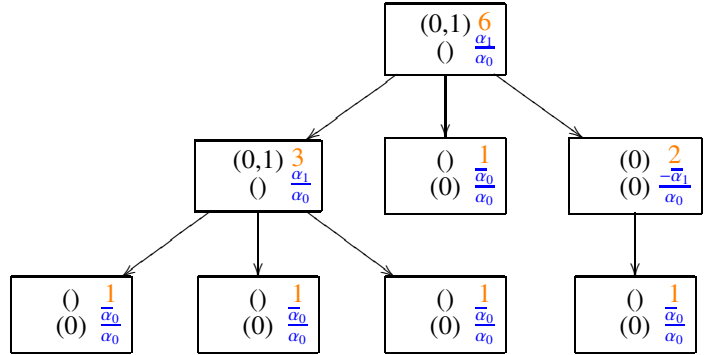
\[
\xymatrix{
  & & *+[F]\txt{\composite{\txt{(0,1)\\ ()}*\txt{\hspace{30pt}\textcolor{orange}{6}\\ \hspace{30pt}\textcolor{blue}{$\frac{\alpha_1}{\alpha_0}$}}}} \ar[d] \ar[ld] \ar[dr] & \\
 & *+[F]\txt{\composite{\txt{(0,1)\\ ()}*\txt{\hspace{30pt}\textcolor{orange}{3}\\ \hspace{30pt} \textcolor{blue}{$\frac{\alpha_1}{\alpha_0}$}}}} \ar[d]\ar[ld]\ar[rd] & *+[F]\txt{\composite{\txt{()\\ (0)}*\txt{\hspace{30pt}\textcolor{orange}{1}\\ \hspace{30pt}\textcolor{blue}{$\frac{\overline{\alpha}_0}{\alpha_0}$}}}} & *+[F]\txt{\composite{\txt{(0)\\ (0)}*\txt{\hspace{30pt}\textcolor{orange}{2}\\ \hspace{30pt}\textcolor{blue}{$\frac{-\overline{\alpha}_1}{\alpha_0}$}}}}\ar[d] \\
 *+[F]\txt{\composite{\txt{()\\ (0)}*\txt{\hspace{30pt}\textcolor{orange}{1}\\ \hspace{30pt}\textcolor{blue}{$\frac{\overline{\alpha}_0}{\alpha_0}$}}}} & *+[F]\txt{\composite{\txt{()\\ (0)}*\txt{\hspace{30pt}\textcolor{orange}{1}\\ \hspace{30pt}\textcolor{blue}{$\frac{\overline{\alpha}_0}{\alpha_0}$}}}} & *+[F]\txt{\composite{\txt{()\\ (0)}*\txt{\hspace{30pt}\textcolor{orange}{1}\\ \hspace{30pt}\textcolor{blue}{$\frac{\overline{\alpha}_0}{\alpha_0}$}}}} & *+[F]\txt{\composite{\txt{()\\ (0)}*\txt{\hspace{30pt}\textcolor{orange}{1}\\ \hspace{30pt}\textcolor{blue}{$\frac{\overline{\alpha}_0}{\alpha_0}$}}}}
}
\]
Each rectangular box represents a vertex, and the black-colored sequences inside the box corresponding to vertex $v$ are the label $(H_v,I_v)$ of $v$. The sequence written in black color on top is $H_v$, while the sequence written in black color on the bottom is $I_v$. To demonstrate how the weights and indices work, we have marked each vertex with its weight in orange color, and with its  $(\alpha_0,\alpha_1)$-index in blue color. We follow the convention from combinatorics that the root vertex is drawn at the top, with the arrows pointing {\em away} from the root vertex\footnote{Although directions on the edges are not included in the definition of a rooted tree, every rooted tree becomes a directed graph in a canonical way, by making each edge point away from the root vertex. This is well-known in combinatorics. From time to time in this paper we will make some statement like this which will be utterly elementary to a combinatorialist, but which we expect will be useful to readers who, like the author, come to this paper from algebraic topology and do not have a great deal of background in combinatorics.}.


In \cref{Drawn examples...}, we give a more complete stock of examples, by drawing every $p$-labelled tree of weight $\leq 4$, for every $p$.

The proof of Theorem \ref{lettered thm A} takes several steps and involves some preliminary results, which appear in \cref{The recursive formula for...}. The idea behind the proof is quite simple, though:
\begin{enumerate}
\item Work out a recursive formula which determines the coefficient $\gamma_n$, from Theorem \ref{lettered thm A}, in terms of $\gamma_i$ for various $i<n$. (We give this recursive formula in Proposition \ref{u1 formula}.) This leans very heavily on the methods of Devinatz and Hopkins from the paper \cite{MR1333942}. Everything we do in the process of working out this recursive formula is a matter of working out the consequences of what Devinatz and Hopkins have already done. 
\item Write $\delta_n$ for the sum of the $(\alpha_0,\alpha_1)$-indices of the weight $n$ $p$-labelled trees. Work out a recursive formula that determines $\delta_n$ in terms of $\delta_i$ for various $i<n$.
\item Observe that $\delta_0 = \gamma_0$ and that the recursion satisfied by $\delta_0,\delta_1,\dots$ is the same recursion satisfied by $\gamma_0,\gamma_1,\dots$. 
\end{enumerate}

Theorem \ref{lettered thm A} is an elementary closed formula for the action of $\Aut(\mathbb{G})$ on $\Def(\mathbb{G}) = W(\mathbb{F}_{p^2})[[u_1]]$. In Theorem \ref{lettered thm B}, we also provide an elementary closed formula for the action of $\Aut(\mathbb{G})$ on the {\em graded} Lubin--Tate ring $\Def_*(\mathbb{G})\cong \Def(\mathbb{G})[u^{\pm 1}]$, by giving a formula for the action of $\Aut(\mathbb{G})$ on $u$:
\begin{letteredtheorem} \label{lettered thm B}
Let $\theta_0,\theta_1, \dots\in W(\mathbb{F}_{p^2})$ be the coefficients in the power series $(\alpha_0 + \alpha_1S).u = \sum_{n\geq 0} \theta_n\cdot u_1^n\cdot u$. Then
\begin{align}
\label{action on u closed formula 3} 
\theta_n &=  \sum_{i=0}^{n}\left( \left(\overline{\alpha}_1 \sum_{\substack{I\in \J^{odd}\\ QI=i}} \frac{1}{p^{\left(\left| I \right| - 1\right)/2}} + \alpha_0 \sum_{\substack{I\in \J^{even}\\ QI=i}} \frac{1}{p^{\left| I \right|/2}}\right) \cdot \sum_{\substack{\text{sequences}\\ \text{$K$ of positive}\\ \text{integers,}\\ \text{ $\Sigma K = n-i$}}} (-1)^{\left| K\right|} \prod_{k\in K} \left(\sum_{\substack{\text{$I\in \J^{even}$}\\\text{$J\in pLT^{QI}$}\\ \text{$\wt(J) = k$}}} \frac{\ind_{\alpha_0,\alpha_1}(J)}{p^{\left| I\right|/2}}\right) 
\right).
\end{align} 
\end{letteredtheorem}
The new notations in \eqref{action on u closed formula 3} are as follows: 
\begin{itemize}
\item $\J^{odd}$ (respectively, $\J^{even}$) is the set of monotone strictly increasing parity-alternating sequences of nonnegative integers, of odd length (respectively, of even length), which do not begin with an odd number,
\item $\Sigma K$ is the sum of the entries in a sequence $K$, 
\item and given a $QI$-tuple $J$ of elements of $pLT$, we write $\wt(J)$ for the sum of weights of the $p$-labelled trees in $J$, and we write $\ind_{\alpha_0,\alpha_1}(J)$ for the product of the $(\alpha_0,\alpha_1)$-indices of the $p$-labelled trees in $J$.
\end{itemize}
Theorem \ref{lettered thm B} is the formal group law case of Theorem \ref{main action thm on u}, a more general result which applies to formal $A$-modules. Most results in this paper are given for formal $A$-modules, with $A$ the ring of integers in a finite field extension of $\mathbb{Q}_p$. The case $A = \hat{\mathbb{Z}}_p$ recovers the case of formal group laws. 

Theorems \ref{lettered thm A} and \ref{lettered thm B}---or better, their formal $A$-module generalizations given below as Theorems \ref{main action thm} and \ref{main action thm on u}---are the main results in this paper, but there are other results scattered throughout which some readers may find useful for some purposes. For example, it was observed by Kohlhaase in \cite{MR3069363} that, at least in low degrees, the action of the subgroup $W(\mathbb{F}_{p^2})^{\times}$ of $\Aut(\mathbb{G})$ on $\Def(\mathbb{G})$ admits a simpler description than the action of $\Aut(\mathbb{G})$ on $\Def(\mathbb{G})$. We show that Kohlhaase's phenomenon holds more generally: Theorem \ref{witt u1 formula} and Corollary \ref{witt action on u thm}, below, specialize Theorems \ref{lettered thm A} and \ref{lettered thm B} to the action of $W(\mathbb{F}_{p^2})^{\times}$ on $\Def_*(\mathbb{G})$, yielding formulas that are indeed notably simpler than those for the full action of $\Aut(\mathbb{G})$. 

Readers who do not want to bother with trees, and just want explicit algebraic formulas in low degrees, may find Corollaries \ref{witt comp cor 1} and \ref{witt comp u cor} to their liking. The proof of Corollary \ref{witt comp cor 1} also contains details and useful methods for counting $p$-labelled trees of specific weights and making explicit calculations using them.

We close with a remark on the possibility of proving analogues of Theorems \ref{lettered thm A} and \ref{lettered thm B} at heights $h>2$. We expect that such generalizations to higher heights ought to be possible, although we haven't tried to work out those generalizations ourselves, and we expect some patience and some work will be required. We do not know what kinds of combinatorial objects would take the place of $p$-labelled trees in height $h>2$ generalizations. 

\subsection{Conventions}
\label{conventions section}
\begin{itemize}
\item In this paper, all formal group laws and formal modules will be implicitly understood to be commutative and one-dimensional.
\item We will often need to refer to a certain power series $f$ introduced by Devinatz and Hopkins in \cite{MR1333942}. We will also sometimes need to refer to the residue degree of an extension of $p$-adic fields, which is also standardly written $f$. To avoid confusion over this clash of notations, we will write $f$ for the Devinatz--Hopkins series, and we will write $\mathfrak{f}$ for the residue degree of a local field extension.
\item We follow a very old-fashioned notational convention \cite{MR0242802} by writing $\hat{\mathbb{Z}}_p$, rather than $\mathbb{Z}_p$, for the ring of $p$-adic integers. 
\item Throughout, we fix a prime $p$ and a finite field extension $F$ of $\mathbb{Q}_p$. We write $\mathcal{O}_F$ for the ring of integers of $F$, and we fix a uniformizer $\pi$ for $\mathcal{O}_F$. We will write $k$ for the residue field $\mathcal{O}_F/\pi$ of $\mathcal{O}_F$, and $q = p^\mathfrak{f}$ for the cardinality of $k$.

If one cares only about the most important case of these ideas, i.e., the base case $F = \mathbb{Q}_p$, then it is not necessary to remember what a uniformizer is, what a ring of integers is, etc. {\bf To the reader who is interested only in formal group laws and not formal modules:} throughout this paper, simply let $F$ be $\mathbb{Q}_p$, let $\mathcal{O}_F = \hat{\mathbb{Z}}_p$, let $\mathfrak{f}=1$, and let $\pi = p = q$. 
\item In this paper, we are primarily concerned with the action of the height $2$ full Morava stabilizer group $\left( W(\mathbb{F}_{p^2})\langle S \rangle/\left(S^2 = p, \overline{\omega}S = S\omega \right)\right)^{\times}$ on the coefficient ring $W(\mathbb{F}_{p^2})[[u_1]][u^{\pm 1}]\cong (E_2)_*$ of height $2$ Morava $E$-theory. The unit group $W(\mathbb{F}_{p^2})^{\times}$ of the Witt ring $W(\mathbb{F}_{p^2})$ has two apparent actions on $W(\mathbb{F}_{p^2})[[u_1]][u^{\pm 1}]$: one by ordinary multiplication in the ring $W(\mathbb{F}_{p^2})$, and one by embedding $W(\mathbb{F}_{p^2})^{\times}$ into the Morava stabilizer group and letting the stabilizer group act on $W(\mathbb{F}_{p^2})[[u_1]][u^{\pm 1}]$. We will use two different notations to distinguish these two actions of an element $\alpha\in W(\mathbb{F}_{p^2})^{\times}$ on an element $x\in W(\mathbb{F}_{p^2})[[u_1]][u^{\pm 1}]$: for the first (i.e., ordinary multiplication in the Witt ring) we will write $\alpha\cdot x$, and for the second (i.e., the Morava action), we will write $\alpha .x$.
\item Throughout, the word ``sequence'' will always mean a sequence {\em of finite length.}
\item We include connectedness as part of the definition of a tree (see Definition \ref{def of rooted tree}).
\end{itemize}

\subsection{Acknowledgments}

This paper has its origins with Mike Hopkins, in 2006, suggesting to the author that the author ought to read \cite{MR1333942} and calculate something with it. In 2024, the author finally read the paper and calculated something with it. The author thanks Mike for conversations about \cite{MR1333942} and for his good suggestion, and apologizes to Mike for a bit of delay in following through.

We are also grateful to Agn\`{e}s Beaudry for answering some questions very helpfully as we worked out early versions of the results in this paper. We also had good conversations with Anish Chedalavada, Austin Maison, and Jared Weinstein about some early cases of these results. We thank them for those conversations. Many of those conversations happened at the December 2024 AIM workshop ``Chromatic homotopy theory and $p$-adic geometry,'' and we are grateful to the organizers of that workshop (Barthel, Schlank, Stapleton, and Weinstein) for inviting the author. We thank the organizers and participants and hosts of the conference for creating an excellent environment for doing mathematics, and for being very patient with the author's many questions.

\section{Review of Devinatz, Gross, and Hopkins' use of Cartier theory at height $2$}
\label{Review of Devinatz...}

Let $p$ be a prime and let $\mathbb{G}$ be the Honda height $2$ formal group law over $\mathbb{F}_{p^2}$. Here are two well-known equivalent characterizations of $\mathbb{G}$:
\begin{enumerate}
\item $\mathbb{G}$ is the $p$-typical formal group law over $\mathbb{F}_{p^2}$ classified by the ring map $BP_*\rightarrow \mathbb{F}_{p^2}$ sending the Hazewinkel generator $v_2$ to $1$, and sending all other Hazewinkel generators $v_n$ of $BP_*$ to zero.  (Here $BP_*$ is the coefficient ring of $p$-primary Brown--Peterson homology, or equivalently, the classifying ring of $p$-typical formal group laws.)
\item $\mathbb{G}$ is the $p$-typical formal group law over $\mathbb{F}_{p^2}$ whose $p$-series is $x^{p^2}$.
\end{enumerate} 
The Lubin--Tate deformation ring $\Def(\mathbb{G})$ of $\mathbb{G}$ is $\Aut(\mathbb{G})$-equivariantly isomorphic to the degree zero homotopy ring of the Morava $E$-theory spectrum $E(\mathbb{G})$:
\begin{align*}
 \pi_0(E(\mathbb{G})) & \cong \Def(\mathbb{G}) \\
 &\cong W(\mathbb{F}_{p^2})[[u_1]].
\end{align*}
Recall that the Lubin--Tate deformation ring $\Def(\mathbb{G})$ is defined by the following property \cite{MR0238854}. Given a Noetherian local $W(\mathbb{F}_{p^2})$-algebra $A$ with residue field $\mathbb{F}_{p^2}$, the set of local\footnote{Recall that a homomorphism between local rings is said to be {\em local} if it sends the maximal ideal of the domain into the maximal ideal of the codomain.} $W(\mathbb{F}_{p^2})$-algebra homomorphisms from $\Def(\mathbb{G})$ to $A$ is in bijection with the set of $\star$-isomorphism classes of formal group laws over $A$ whose reduction modulo the maximal ideal is equal to $\mathbb{G}$. A {\em $\star$-isomorphism} between two such formal group laws $F_1,F_2$ is an isomorphism of formal group laws $F_1\stackrel{\cong}{\longrightarrow} F_2$ over $A$ which reduces, modulo the maximal ideal of $A$, to the identity map.

In the {\em graded} version of Lubin--Tate theory (see \cite[section 1]{MR1333942}), one lets $\Def_*(\mathbb{G})$ be the classifying ring of $\star$-isomorphism classes of pairs $(F,a)$, where $F$ is a formal group law over $A$ and $a$ is a unit in $A$. In the graded context, such pairs $(F_1,a_1), (F_2,a_2)$ are said to be {\em $\star$-isomorphic} if there is an (ungraded) $\star$-isomorphism $\phi: F_1 \stackrel{\cong}{\longrightarrow} F_2$ such that $a_1 = \phi^{\prime}(0)\cdot a_2$. (Note that $\star$-isomorphisms are not required to be strict isomorphisms, i.e., $\phi^{\prime}(0)$ might not be equal to $1$.) The ring $\Def_*(\mathbb{G})$ is $\Aut(\mathbb{G})$-equivariantly isomorphic to the graded coefficient ring $\pi_*(E(\mathbb{G}))\cong W(\mathbb{F}_{p^2})[[u_1]][u^{\pm 1}]$ of the Morava $E$-theory associated to $\mathbb{G}$. The grading is such that $u_1$ is in degree zero, and $u$ is in degree $-2$.

Returning to the ungraded story, the ring $\Def(\mathbb{G})$ has a natural $BP_*$-algebra structure given by classifying the $p$-typicalization $\hat{\mathbb{G}}$ of the universal deformation of $\mathbb{G}$. This $BP_*$-algebra structure is given by the ring map $BP_* \rightarrow \Def(\mathbb{G})$ sending $v_1$ to $u_1$, sending $v_2$ to $1$, and sending the other Hazewinkel generators $v_n$, $n>2$, to zero. Since $\hat{\mathbb{G}}$ is $p$-typical, its log series $\log_{\hat{\mathbb{G}}}(X)$ is of the form $\sum_{n\geq 0} m_n X^{p^n}$ for some $m_0,m_1, m_2, \dots \in \Def(\mathbb{G})\otimes_{\mathbb{Z}}\mathbb{Q}$ with $m_0=1$. 

The remarkable observation which gets the Devinatz--Hopkins paper \cite{MR1333942} started is that, when properly normalized, the sequence of odd log coefficients $m_1, m_3, m_5, \dots$ actually {\em converges} to some power series in the variable $u_1$, and similarly, normalized versions of the even log coefficients $m_0, m_2, m_4, \dots$ also converge. Each $m_n$ is an element\footnote{In fact, Devinatz--Hopkins prove that each $m_n$ is also an element in the smaller ring $W(\mathbb{F}_{p^2})\langle\langle u_1\rangle\rangle \subseteq \Def(\mathbb{G}) \otimes_{\mathbb{Z}}\mathbb{Q}$, the {\em divided power envelope} of $\Def(\mathbb{G})$, but in this paper we will make so little use of this fact that the reader need not know the definition of the divided power envelope. For our purposes it is enough to know that we have proper containments of subrings
$\Def(\mathbb{G}) \subset W(\mathbb{F}_{p^2})\langle\langle u_1\rangle\rangle\subset \Def(\mathbb{G}) \otimes_{\mathbb{Z}}\mathbb{Q}$. It is important, though, that the elements $m_2,m_4,m_6,\dots$ do not live in $\Def(\mathbb{G})$ itself, even after normalization, and neither does the power series which they converge to. The same claim holds for the odd log coefficients $m_1, m_3, \dots$.} in $\Def(\mathbb{G}) \otimes_{\mathbb{Z}}\mathbb{Q}$. Devinatz and Hopkins observe \cite[formula (4.3) in section 4]{MR1333942} that, as a simple consequence of the formula defining the Hazewinkel generators for $BP_*$, the elements $m_0,m_1,m_2, \dots$ satisfy the Fibonacci-like recursion
\begin{align}
\label{dh recursion} m_n &= \frac{u_1^{p^{n-1}}}{p} m_{n-1} + \frac{1}{p} m_{n-2} \mbox{\ \ if\ } n\geq 2,\\
\nonumber m_0 &= 1,\\
\nonumber m_1 &= \frac{u_1}{p}.
\end{align}
Devinatz and Hopkins write $f_1$ for the power series $\lim_{n\rightarrow \infty} p^{n+1}m_{2n+1}$, and $f$ for the power series $\lim_{n\rightarrow \infty} p^n m_{2n}$. The elements $w := u\cdot f$ and $ww_1 := u\cdot f_1$ of $\mathbb{Q}\otimes_{\mathbb{Z}}\pi_{-2}(E(\mathbb{G})) \cong \mathbb{Q}\otimes_{\mathbb{Z}} W(\mathbb{F}_{p^2})[[u_1]]\{ u\}$ are called {\em Cartier coordinates}\footnote{To be clear about notation: we write $W(\mathbb{F}_{p^2})[[u_1]]\{ u\}$ to mean the free $W(\mathbb{F}_{p^2})[[u_1]]$-module on a single generator named $u$.}. The merit of the Cartier coordinates is that the action of $\Aut(\mathbb{G})$ is {\em linear} on the Cartier coordinates. More precisely, by \cite[Proposition 3.3 and Remark 2.21]{MR1333942}, given the element 
\begin{align*}
 \alpha_0 + \alpha_1S &\in \left( W(\mathbb{F}_{p^2})\langle S \rangle/\left(S^2 = p, \overline{\omega}S = S\omega \mbox{\ for\ each\ }\omega\in W(\mathbb{F}_{p^2})\right)\right)^{\times} \cong \Aut(\mathbb{G}),\end{align*}
the action of $\Aut(\mathbb{G})$ on the Cartier coordinates is given by the formulas
\begin{align}
\label{action formula 1} (\alpha_0 + \alpha_1S)\act ww_1 &= \overline{\alpha}_0 ww_1 + p\alpha_1w \\
\label{action formula 2} (\alpha_0 + \alpha_1S)\act w &= \overline{\alpha}_1 ww_1 + \alpha_0w .
\end{align}
The action of $\Aut(\mathbb{G})$ on $\Def_*(\mathbb{G})$ is then determined by the formulas \eqref{action formula 1} and \eqref{action formula 2} together with the fact that the action of $\Aut(\mathbb{G})$ on $\Def_*(\mathbb{G})$ is continuous and by $W(\mathbb{F}_{p^2})$-algebra automorphisms.

Hence one has, at least in principle, a means of making explicit calculation of the action of $\Aut(\mathbb{G})$ on $\Def_*(\mathbb{G})\cong W(\mathbb{F}_{p^2})[[u_1]][u^{\pm 1}]$: convert from the Hazewinkel coordinates $u_1$ and $u$ to the Cartier coordinates $w_1$ and $w$, apply the action of $\Aut(\mathbb{G})$ on the Cartier coordinates, then convert back to Hazewinkel coordinates. 

The papers of Chai \cite{MR1387691} and Kohlhaase \cite{MR3069363} use the Cartier-theoretic methods of Devinatz, Gross, and Hopkins to make their approximate calculations of the action of $\Aut(\mathbb{G})$ on $\Def(\mathbb{G})$ described in \cref{History of the problem}. Other existing works \cite{karamanovthesis}, \cite{MR3127044}, \cite{laderthesis}, \cite{MR4435205} use a different method (essentially using the structure maps of the classifying Hopf algebroid of $p$-typical formal group laws and not-necessarily-strict isomorphisms) which does not involve Cartier coordinates. 

We will work in slightly greater generality than Devinatz and Hopkins do: we fix a finite field extension $F/\mathbb{Q}_p$, write $A$ for its ring of integers, $\pi$ for a fixed choice of uniformizer for $A$, and $q$ for the cardinality of the residue field. For such rings $A$, the moduli theory of formal $A$-modules was shown by Drinfeld \cite{MR0384707} to closely resemble that of formal group laws, and indeed, in the case $F = \mathbb{Q}_p$ (and consequently $A = \hat{\mathbb{Z}}_p$ and $\pi = p=q$), formal group laws over commutative $\hat{\mathbb{Z}}_p$-algebras are the same thing as formal $\hat{\mathbb{Z}}_p$-modules. A formal $A$-module $F$ with logarithm is said to be {\em $A$-typical} if its logarithm is of the form 
\begin{equation}\label{A-typ log}
X + \alpha_1X^1 + \alpha_2X^{q^2} + \alpha_3X^{q^3} + \dots\end{equation} 
for some rational scalars $\alpha_1,\alpha_2,\alpha_3$. We say that $F$ has {\em $A$-height $\geq n$} if $\alpha_1, \dots ,\alpha_{n-1}$ are all zero, i.e., the logarithm \eqref{A-typ log} is congruent to $X + \alpha_nX^{q^n}$ modulo $X^{q^n+1}$. 
There exists \cite{MR2987372} a classifying ring $V^A\cong A[v_1^A,v_2^A,v_3^A,\dots]$ of $A$-typical formal $A$-modules; when $A = \hat{\mathbb{Z}}_p$, it is the tensor product $BP_*\otimes_{\mathbb{Z}}\hat{\mathbb{Z}}_p$, where $BP_*$ is the familiar Brown--Peterson coefficient ring, which classifies $p$-typical formal group laws over commutative $\mathbb{Z}_{(p)}$-algebras. As expected, a formal $A$-module over a commutative $A$-algebra $R$ has $A$-height $\geq n$ if and only if the classifying map $V^A \rightarrow R$ sends $v_1^A,v_2^A, \dots ,v_{n-1}^A$ all to zero.

\section{From Hazewinkel coordinates to Cartier coordinates}

In this section we calculate the Cartier coordinate $w_1$ in terms of the Hazewinkel coordinate $u_1$. We present a self-contained calculation, but the result is not new: an equivalent calculation, yielding the same result (though stated in somewhat different terms), appeared already\footnote{For comparison with \cite{MR1263712}, we note that the power series called $f$ and $f_1$ in \cite{MR1333942} and in the present paper are instead called $\phi_0$ and $\phi_1$ in \cite{MR1263712}.} in \cite[section 25]{MR1263712}, and more generally, at all heights in \cite[section 11 proposition 8]{MR1327148}. 

Henceforth all our results and arguments in this paper will be given for formal $\mathcal{O}_F$-modules, not only formal group laws. See \cref{conventions section} for relevant conventions and notations. The formal module version of the Devinatz--Gross--Hopkins theory is certainly not new: it was studied already by Gross and Hopkins in \cite{MR1263712}, and later by Chai in \cite{MR1387691}, Yu in \cite{MR1327148}, and Kohlhaase in \cite{MR3069363}.
In the setting of formal $\mathcal{O}_F$-modules, the recursion \eqref{dh recursion} becomes
\begin{align}
\label{dh recursion 2} m_n &= \frac{u_1^{q^{n-1}}}{\pi} m_{n-1} + \frac{1}{\pi} m_{n-2} \mbox{\ \ if\ } n\geq 2,\\
\nonumber m_0 &= 1,\\
\nonumber m_1 &= \frac{u_1}{\pi}.
\end{align}
It will be convenient to describe the recursion \eqref{dh recursion 2} in a slightly different way. Let $\mathcal{M}_{2n}$ denote $\pi^{n}m_{2n}$, and let $\mathcal{M}_{2n+1}$ denote $\pi^{n+1}m_{2n+1}$. Then the recursion \eqref{dh recursion} is equivalent to the equalities
\begin{align}
\label{recursion 2} \mathcal{M}_{2n} &= \mathcal{M}_{2n-2} + \frac{u_1^{q^{2n-1}}}{\pi}\mathcal{M}_{2n-1} \mbox{\ \ if\ } n\geq 1,\\
\nonumber \mathcal{M}_{2n+1} &= u_1^{q^{2n}}\mathcal{M}_{2n-2} + \left( 1 + \frac{u_1^{q^{2n-1}+q^{2n}}}{\pi}\right)\mathcal{M}_{2n-1} \mbox{\ \ if\ } n\geq 1,\\
\nonumber \mathcal{M}_0&= 1,\\
\nonumber \mathcal{M}_1&= u_1.
\end{align}
Now, for each $n=0,1,2,\dots$, let $F_n$ denote the power series $\mathcal{M}_{2n+1}/\mathcal{M}_{2n}$. The Cartier coordinate $w_1$ is then simply the limit $\lim_{n\rightarrow\infty} F_n$. By routine algebra, the recursion \eqref{recursion 2} yields a recursion for the sequence $F_0,F_1,\dots$:
\begin{align*}
 F_{n+1} &= \frac{(u_1^{q^{2n+2} + q^{2n+1}} + \pi)F_n + \pi u_1^{q^{2n+2}}}{u_1^{q^{2n+1}}F_n + \pi} \mbox{\ \ if\ }n\geq 0,\\
 F_0 &= u_1.
\end{align*}
In other words, if we consider $PSL_2(W(\mathbb{F}_{p_2}[[u_1]])$ acting on $W(\mathbb{F}_{p_2}[[u_1]])\otimes_{\mathbb{Z}}\mathbb{Q}$ by fractional linear transformations via the usual formula
\begin{align*}
\begin{bmatrix} 
 a & b \\
 c & d
\end{bmatrix} \cdot F &= \frac{aF + b}{cF + d},
\end{align*}
then the Cartier coordinate $w_1$ is described in terms of the Hazewinkel coordinate $u_1$ as an infinite product
\begin{align}
\label{w1 formula 0} w_1 &= \lim_{n\rightarrow \infty} \left( \mathcal{T}_n \mathcal{T}_{n-1} \dots \mathcal{T}_1 \cdot u_1\right),
\end{align}
where $\mathcal{T}_i$ is the matrix
\begin{align*}
 \mathcal{T}_i &= \begingroup
\renewcommand*{\arraystretch}{2.3}
\begin{bmatrix}
 1 + \frac{u_1^{q^{2i} + q^{2i-1}}}{\pi} & u_1^{q^{2i}} \\ 
 \frac{u_1^{q^{2i-1}}}{\pi} & 1 
\end{bmatrix} .
\endgroup
\end{align*}

\begin{definition}\label{def of J}\leavevmode
\begin{itemize}
\item
We will say that a monotone strictly increasing sequence of integers $(i_1, i_2, \dots , i_{\ell})$ is {\em parity-alternating} if the elements $i_1, \dots ,i_{\ell-1}$ are each of opposite parity from the next element in the sequence. 
\item 
We write $\left| I\right|$ for the length of a sequence $I$, i.e., $\left| (i_1, i_2, \dots , i_{\ell})\right| = \ell$. 
\item 
We will write $\J$ for the set of monotone strictly increasing, parity-alternating sequences\footnote{We remind the reader of the convention, from Conventions \ref{conventions section}, that {\em all sequences in this paper are of finite length.}} of nonnegative integers which do not begin with an odd number\footnote{To be absolutely clear: each member of $\J$ is either empty or has even first entry. For example, the members of $\J$ include sequences like $(0,1)$ and $(0,1,4)$ and $(2,3,6)$ and the empty sequence $()$. The set $\J$ does not include the sequence $(1,4)$, whose first entry is odd, or the sequence $(0,4)$, which is not parity-alternating.}
.
\item
We will write $\J^{odd}$ for the subset of $\J$ consisting of the sequences of odd length, and we will write $\J^{even}$ for the subset of $\J$ consisting of the sequences of even length.
\item We will write $\Sigma I$ for the sum of the elements of a sequence $I$, i.e., $\Sigma (i_1, \dots ,i_{\ell}) = \sum_{j=1}^{\ell} i_j$.
\item Given a prime power $q$ and a sequence $I = (i_1, \dots ,i_{\ell})$, we will write $QI$ for the sum $\sum_{j=1}^{\ell} q^{i_j}$.
\end{itemize}
\end{definition}

\begin{theorem}\label{w1 calc}
The Cartier coordinate $w_1\in W(\mathbb{F}_{p^2})[[u_1]]\otimes_{\mathbb{Z}}\mathbb{Q}$ is equal to the quotient power series 
\begin{align}
\label{w1 formula} w_1 &=  \frac{\sum_{I\in \J^{odd}} u_1^{QI}/\pi^{\left(\left|I\right| - 1\right)/2}}{\sum_{I\in \J^{even}} u_1^{QI}/\pi^{\left|I\right|/2}}.\end{align}
\end{theorem}
\begin{proof}
By a routine induction, one shows that the product $\mathcal{T}_n \dots \mathcal{T}_2\mathcal{T}_1$ is equal to the matrix $\begin{bmatrix} a_n & b_n \\ c_n & d_n\end{bmatrix}$, where:
\begin{itemize}
\item $a_n$ is the sum of $u_1^{QI}/\pi^{\left|I\right|/2}$ over all even-length, parity-alternating, strictly monotone increasing sequences $I$ whose first element is odd and whose members are drawn from the set $\{ 1, \dots , 2n\}$,
\item $b_n$ is the sum of $u_1^{QI}/\pi^{\left(\left|I\right|-1\right)/2}$ over all odd-length, parity-alternating, strictly monotone increasing sequences $I$ whose first element is even and whose members are drawn from the set $\{ 1, \dots , 2n\}$,
\item $c_n$ is the sum of $u_1^{QI}/\pi^{\left(\left|I\right|+1\right)/2}$ over all odd-length, parity-alternating, strictly monotone increasing sequences $I$ whose first element is odd and whose members are drawn from the set $\{1, \dots ,2n-1\}$,
\item and $d_n$ is the sum of $u_1^{QI}/\pi^{\left|I\right|/2}$ over all even-length, parity-alternating, strictly monotone increasing sequences $I$ whose first element is even and whose members are drawn from the set $\{ 1, \dots ,2n-1\}$.
\end{itemize}
Formula \eqref{w1 formula} then follows from routine algebra together with formula \eqref{w1 formula 0}.
\end{proof}
For example, modulo $u_1^{q^5}$, we have
\begin{align*}
 w_1 &\equiv \frac{(u_1^{1+q+q^2+q^3+q^4})/\pi^2 + (u_1^{1+q+q^2} + u_1^{1+q+q^4} + u_1^{1+q^3+q^4} + u_1^{q^2+q^3+q^4})/\pi + (u_1 + u_1^{q^2} + u_1^{q^4})}{(u_1^{1+q+q^2+q^3})/\pi^2 + (u_1^{1+q} + u_1^{1+q^3} + u_1^{q^2+q^3})/\pi + 1} .
\end{align*}

One can also calculate the power series $f_1$ and $f$ of Devinatz, Gross, and Hopkins, defined in the paragraph following \eqref{dh recursion}, directly (i.e., not simply the quotient $w_1 = f_1/f$) using the same recursive method as used in the proof of Theorem \ref{w1 calc}, yielding:
\begin{proposition}\label{f1 and f formula}
The power series $f_1$ and $f$ are equal to the numerator and the denominator, respectively, of the right-hand side of \eqref{w1 formula}.
\end{proposition}

\section{The recursive formula for the action of $\Aut(\mathbb{G})$ on $u_1$}
\label{The recursive formula for...}

Given an element $\omega \in W(\mathbb{F}_{q^2})$, we write $\sigma(\omega)$ or $\overline{\omega}$ for the image of $\omega$ under the canonical lift of the nontrivial Galois automorphism of $\mathbb{F}_{q^2}$ over $\mathbb{F}_q$ to an automorphism of the Witt ring $W(\mathbb{F}_{q^2})$. 
There is a well-known isomorphism (e.g. see \cite[Proposition 1.7(2)]{MR0384707}):
\begin{align*} \Aut(\mathbb{G}) &\cong \left(\mathcal{O}_F\otimes_{W(\mathbb{F}_q)} W(\mathbb{F}_{q^2})\langle S \rangle/\left(S^2 = \pi,\overline{\omega} S = S\omega \ \ \forall \omega\in W(\mathbb{F}_{q^2})\right)\right)^{\times}.\end{align*}
Consequently any given element of $\Aut(\mathbb{G})$ has a unique expression of the form $\alpha_0 + \alpha_1S$ with $\alpha_0,\alpha_1\in W(\mathbb{F}_{q^2})$ and with $\alpha_0$ a unit in $W(\mathbb{F}_{q^2})$.
Since the action of $\Aut(\mathbb{G})$ on $\Def(\mathbb{G}) \cong W(\mathbb{F}_{q^2})[[u_1]]$ is by continuous $W(\mathbb{F}_{q^2})$-algebra automorphisms, if we can give a formula for the action of $a_0+a_1S\in \Aut(\mathbb{G})$ on $u_1$, this extends in a straightforward way to a formula for the action of $\Aut(\mathbb{G})$ on any element of $\Def(\mathbb{G})$. The main result of this section is Proposition \ref{u1 formula}, which is a {\em recursive} formula of this kind. We will not get to a {\em closed} formula of this kind until Theorem \ref{main action thm}, but the proof of that theorem requires the recursion given by Proposition \ref{u1 formula}.

Proposition \ref{u1 formula} uses some notations, namely $\left| K\right|,\Sigma K,QK$, and $\J$, which were defined in Definition \ref{def of J}. We also use the common notation $\lfloor n\rfloor$ for the integer floor of a rational number $n$.
\begin{proposition}\label{u1 formula}
\begin{align}\label{u1 formula a}
 (\alpha_0 + \alpha_1S) . u_1 
  &= \sum_{n\geq 1} \gamma_n u_1^n, 
\end{align}
where $\gamma_0,\gamma_1,\gamma_2,\dots\in W(\mathbb{F}_{q^2})$ are given recursively by the formula 
\begin{align}
\label{gamman formula} \gamma_n &= \frac{1}{\alpha_0} 
\sum_{\substack{\text{sequences $K$ of}\\ \text{positive integers,}\\ K\neq (n)}}
\left( \prod_{k\in K} \gamma_k\right) \sum_{\substack{H,I\in \J : \\ QH  = \left| K \right| ,\\ QI = n - \Sigma K }} (-1)^{\left| H\right|} \frac{\sigma^{\left| I\right|} (\alpha_{\left|H\right| + \left|I\right| - 1})}{\pi^{\lfloor(\left|H\right| + \left|I\right|-1)/2\rfloor}}.
\end{align}
The subscript of $\alpha_{\left| H \right| + \left| I\right|-1}$ is to be understood as being defined modulo $2$.
\end{proposition}
For any fixed choice of integer $n$, the condition $QI = n - \Sigma K$ in \eqref{gamman formula} ensures that the sum is taken only over sequences $K$ whose entries are less than $n$. Hence the outermost sum in \eqref{gamman formula} can equally well be regarded as the sum over all finite-length sequences $K$ of integers taken from the set $\{1, \dots ,n-1\}$. 

To be clear, the sum in \eqref{gamman formula} {\em does} include the case of the empty sequence $K = ()$. Before proving Proposition \ref{u1 formula}, it helps build intuition and familiarity if we consider how the formula \eqref{gamman formula} plays out for low values of $n$:
\begin{description}
\item[If $n=0$] then the sum in \eqref{gamman formula} is empty, i.e., $\gamma_0 = 0$.
\item[If $n=1$] then the sum in \eqref{gamman formula} has a single summand, corresponding to the triple of sequences $K = ()$, $H = ()$, $I = (0)$. The resulting equation is $\gamma_1 = \overline{\alpha}_0/\alpha_0$.
\item[If $n=2$] then the sum in \eqref{gamman formula} has a single summand, corresponding to the triple of sequences $K=(1),H=(0),I=(0)$. The resulting equation is $\gamma_2 = -\gamma_1\frac{\overline{\alpha}_1}{\alpha_0} = \frac{-\overline{\alpha}_0\cdot\overline{\alpha}_1}{\alpha_0^2}$.
\item[If $n=3$] then the sum in \eqref{gamman formula} has summands corresponding to the following triples of sequences:
\begin{itemize}
\item $K = (2),H=(0),I = (0)$,
\item and the triples $K=(),H=(),I=(0,1)$ and $K = (1,1,1),H=(0,1),I=()$ if and only if $q=2$.
\end{itemize}
The resulting equation is 
\begin{align}
\nonumber \gamma_3 &= \left\{\begin{array}{ll} 
  -\gamma_2\frac{\overline{\alpha}_1}{\alpha_0} & \mbox{\ if\ } q\neq 2 \\
  -\gamma_2\frac{\overline{\alpha}_1}{\alpha_0} + \frac{\alpha_1}{\alpha_0} + \gamma_1^3\frac{\alpha_1}{\alpha_0} & \mbox{\ if\ } q=2 \\
\end{array}\right.\\
\label{eq 1901} &= \left\{\begin{array}{ll} 
  \frac{\overline{\alpha}_0\cdot \overline{\alpha}_1^2}{\alpha_0^3} & \mbox{\ if\ } q\neq 2 \\
 \frac{N(\alpha_0)\cdot \overline{\alpha}_1^2 + (\alpha_0^3 + \overline{\alpha}_0^3)\cdot \alpha_1}{\alpha_0^4} & \mbox{\ if\ } q=2,
\end{array}\right.
\end{align}
where $N$ denotes the norm $N(x) = x\cdot \overline{x}$ in $W(\mathbb{F}_{q^2})$.
\item[If $n=4$] then the sum in \eqref{gamman formula} has summands corresponding to the following triples of sequences:
\begin{itemize}
\item $K = (3),H=(0),I = (0)$, 
\item $K = (),H=(),I=(0,1)$ and $K=(1,1,1,1),H=(0,1),I=()$ if and only if $q=3$,
\item and the following triples if and only if $q=2$:
\begin{itemize}
\item $K = (1),H=(0),I = (0,1)$,
\item $K=(1,1,2),H=(0,1),I=()$ and $K=(1,2,1),H=(0,1),I=()$ and $K=(2,1,1),H=(0,1),I=()$, which are permutations of one another under the evident action of the symmetric group on $K$,   
\item $K=(1,1,1),H=(0,1),I=(0)$,
\item $K=(1,1,1,1),H=(2),I=()$,
\item and $K= (), H=(),I=(2)$.
\end{itemize}
\end{itemize}
The resulting equation is 
\begin{align*}
 \gamma_4 &= \left\{\begin{array}{ll} 
  -\gamma_3\frac{\overline{\alpha}_1}{\alpha_0} & \mbox{\ if\ } q\neq 2,3, \\
  -\gamma_3\frac{\overline{\alpha}_1}{\alpha_0} + \frac{\alpha_1}{\alpha_0} + \gamma_1^4\frac{\alpha_1}{\alpha_0} & \mbox{\ if\ } q=3, \\
  -\gamma_3\frac{\overline{\alpha}_1}{\alpha_0} - \gamma_1\frac{1}{\pi} + 3\gamma_1^2\gamma_2\frac{\alpha_1}{\alpha_0} + \gamma_1^3\frac{\overline{\alpha}_0}{\pi \alpha_0} - \gamma_1^4 + \frac{\overline{\alpha}_0}{\alpha_0} & \mbox{\ if\ } q=2 
\end{array}\right.\\
&=\left\{\begin{array}{ll} 
  -\frac{\overline{\alpha}_0\cdot \overline{\alpha}_1^3}{\alpha_0^4}  & \mbox{\ if\ } q\neq 2,3, \\
  \frac{(\alpha_0^4 + \overline{\alpha}_0^4)\alpha_1 - N(\alpha_0)\cdot \overline{\alpha}_1^3}{\alpha_0^5} & \mbox{\ if\ } q=3, \\
   \frac{N(\alpha_0)\cdot \left(-\overline{\alpha}_1^3 + (1-\frac{1}{\pi})(\alpha_0^3 - \overline{\alpha}_0^3)\right) - N(\alpha_1)\cdot \left( \alpha_0^3 + 4\overline{\alpha}_0^3\right)}{\alpha_0^5} & \mbox{\ if\ } q=2 .
\end{array}\right.
\end{align*}
\end{description}
\begin{proof}[Proof of Proposition \ref{u1 formula}]
Write $\Gamma(u_1)$ for the power series $\sum_{n\geq 1}\gamma_n u_1^n$ of \eqref{u1 formula a}. In terms of the Cartier coordinate $w_1$, we have
\begin{align*}
 (\alpha_0 + \alpha_1S). w_1 &= \frac{\overline{\alpha}_0w_1 + \pi \alpha_1}{\overline{\alpha}_1w_1 + \alpha_0},
\end{align*}
by the formal $A$-module analogue of formula \eqref{action formula 1}. 

Recall from Proposition \ref{f1 and f formula} that the power series $f_1(u_1)$ is equal to the numerator $\sum_{I\in \J^{odd}} u_1^{QI}/\pi^{\left(\left|I\right| - 1\right)/2}$ of the right-hand side of formula \eqref{w1 formula}, while the power series $f(u_1)$ is equal to the denominator $\sum_{I\in \J^{even}} u_1^{QI}/\pi^{\left|I\right|/2}$.
We have the equalities of power series
\begin{align}
\nonumber \frac{f_1(\Gamma(u_1))}{f(\Gamma(u_1))}
  &= \frac{\overline{\alpha}_0 w_1 + \pi \alpha_1}{\overline{\alpha}_1w_1 + \alpha_0}\\
\nonumber  &= \frac{\overline{\alpha}_0f_1 + \pi \alpha_1f}{\overline{\alpha}_1f_1 + \alpha_0f},\mbox{\ \ \ i.e.,} \\
\label{eq 330249} f_1(\Gamma(u_1))\cdot (\overline{\alpha}_1f_1 + \alpha_0f)
  &= f(\Gamma(u_1)) \cdot (\overline{\alpha}_0f_1 + \pi \alpha_1f).
\end{align}
Expanding each side of \eqref{eq 330249} using Proposition \ref{f1 and f formula} yields the next equality:
\begin{align}
\nonumber 0 &= \sum_{n\geq 1} u_1^n \sum_{\substack{\text{sequences $K$ of}\\ \text{positive integers}}} \left( \prod_{k\in K} \gamma_k\right) \left( 
   \sum_{\substack{H,I\in \J:\\ \left| H\right|\text{\ is\ odd},\\ \left| I\right|\text{\ is\ even}, \\ \left|K\right| = QH,\\ QI + \Sigma K = n}} \frac{\alpha_0}{\pi^{(\left| H\right| + \left| I\right| - 1)/2}}
 - \sum_{\substack{H,I\in \J:\\ \left| H\right|\text{\ is\ even},\\ \left| I\right|\text{\ is\ odd}, \\ \left|K\right| = QH,\\ QI + \Sigma K = n}} \frac{\overline{\alpha}_0}{\pi^{(\left| H\right| + \left| I\right| - 1)/2}} \right. \\ 
\nonumber  & \ \ \ \ \ \ \ \ \ \ \ \ \ \ \ \ \ \ \ \ \ \ \ \ \ \ \ \ \ \ \ \ \ \ \ \ \ \ \ \ \ \ \ \left. + \sum_{\substack{H,I\in \J:\\ \left|H\right|,\left| I\right|\text{\ are\ odd}, \\ \left|K\right| = QH,\\ QI + \Sigma K = n}} \frac{\overline{\alpha}_1}{\pi^{(\left| H\right| + \left| I\right|)/2 - 1}}
 - \sum_{\substack{H,I\in \J:\\ \left|H\right|,\left| I\right|\text{\ are\ even}, \\ \left|K\right| = QH,\\ QI + \Sigma K = n}} \frac{\alpha_1}{\pi^{(\left| H\right| + \left| I\right|)/2 - 1}} \right) \\
\nonumber &=  \sum_{n\geq 1} u_1^n \sum_{\substack{\text{sequences $K$ of}\\ \text{positive integers}}} \left( \prod_{k\in K} \gamma_k\right) \sum_{\substack{H,I\in \J: \\ \left|K\right| = QH,\\ QI + \Sigma K = n}} (-1)^{1 + \left| H\right|} \frac{\sigma^{\left| I\right|}(\alpha_{\left| H \right| + \left| I\right|-1})}{\pi^{\lfloor\left( \left| H \right| + \left| I \right| - 1\right)/2\rfloor}}, \mbox{\ \ \ i.e.,}\\
\label{eq 330251}
 0 &=  \sum_{\substack{\text{sequences $K$ of}\\ \text{positive integers}}} \left( \prod_{k\in K} \gamma_k\right) \sum_{\substack{H,I\in \J: \\ \left|K\right| = QH,\\ QI + \Sigma K = n}} (-1)^{1 + \left| H\right|} \frac{\sigma^{\left| I\right|}(\alpha_{\left| H \right| + \left| I\right|-1})}{\pi^{\lfloor\left( \left| H \right| + \left| I \right| -1\right)/2\rfloor}}, \mbox{\ \ \ for\ each\ } n.
\end{align}

One of the summands in the sum \eqref{eq 330251}, corresponding to $K = (n)$ and $H = (0)$ and $I = ()$, is $\gamma_n  \alpha_0$. Solving \eqref{eq 330251} for $\gamma_n$ yields the formula \eqref{gamman formula}.
\end{proof}

\section{The closed formula for the action of $\Aut(\mathbb{G})$ on $u_1$}

In this section we improve on the recursive formula given in Proposition \ref{u1 formula} by obtaining a {\em closed} formula for the action of $\Aut(\mathbb{G})$ on $u_1$. There is a small price to be paid for a closed formula, however: we must accept a sum over certain labelled ordered rooted trees, in the sense of combinatorics. We review the relevant definitions from combinatorics (e.g. see \cite[Appendix]{MR2868112}) a bit more leisurely than we did in \cref{The main results}. 
\begin{definition}\label{def of rooted tree}
An {\em ordered rooted tree} is a tree (i.e., a connected acyclic graph) $T$ equipped with a choice of vertex, called the {\em root}, and an ordered partition of the remaining vertices into $n\geq 0$ disjoint nonempty subgraphs $T_1, \dots ,T_n$, each of which is equipped with the structure of a rooted tree. 
The root vertices of the subgraphs $T_1, \dots ,T_n$ are called {\em immediate descendants} of the root of $T$.
For each vertex $v\in T$, there is a unique ordered rooted subtree\footnote{To be clear, by a ``rooted subtree'' of $T$ we mean a subtree $S$ of $T$ equipped with a root---not necessarily the same vertex as the root of $T$!---such that, if we equip $S$ and $T$ with directions on the edges in the usual way (i.e., all edges point away from the root), then the direction on each edge in $S$ agrees with its direction in $T$.} $T_v$ of $T$ of which $v$ is the root. The vertices of $T_v$ are called the {\em descendants} of $v$; they are simply the vertices $w\in T$ such that every path from the root of $T$ to $w$ passes through $v$. (Note that, in particular, $v$ is a descendant of $v$.)
\end{definition}
Explicit examples of ordered rooted trees are depicted in \cref{Drawn examples...}.

We will need several new definitions, Definitions \ref{def of q-labelling} and Definition \ref{def of index}. In these definitions, we will freely refer to some notations introduced at the beginning of \cref{The recursive formula for...}: given a sequence $I = (i_1, \dots, i_m)$, we write $\left| I\right|$ for the length of the sequence, i.e., $m$. If we are also given a prime power $q$, we write $QI$ for the sum $q^{i_1} + \dots + q^{i_m}$. 
\begin{definition}\label{def of q-labelling}
Given a prime power $q$ and an ordered rooted tree $T$, a {\em $q$-pre-labelling of $T$} is the following data. For each vertex $v$ of $T$, we specify an ordered pair of elements $(H_v,I_v)$ of $\J$, called the {\em label of $v$}, satisfying the following condition:
\begin{enumerate}
\item The number of immediate descendants of $v$ is equal to $QH_v$.
\end{enumerate}

Given a $q$-pre-labelling of $T$, the {\em weight} of a vertex $v\in T$ is defined as the sum 
\begin{align*}
 \wt(v) &= \sum_{\substack{\text{descendants}\\w\text{\ of\ } v}} QI_w.
\end{align*}

A $q$-pre-labelling of $T$ is called a {\em $q$-labelling} if the following condition is also satisfied:
\begin{enumerate}
\item[(2)] The weight of each vertex is positive, and strictly greater than the weight of each of its immediate desendants.
\end{enumerate}

The weight $\wt(T)$ of the $q$-labelled tree $T$ itself is defined to be the weight of the root of $T$.

We write $qLT$ for the set of finite ordered rooted $q$-labelled trees.
\end{definition}

\begin{definition}\label{def of index}
Suppose we are given an ordered pair of elements $(\alpha_0,\alpha_1)$ in an integral domain $R$ equipped with a choice of element $\pi\in R$ and an involution $\sigma$, with $\alpha_0\neq 0$.
The {\em $(\alpha_0,\alpha_1)$-index} of a $q$-labelled ordered rooted tree $T$ is defined as the product:
\begin{align}
\label{index def} \ind_{\alpha_0,\alpha_1}(T) &= \frac{1}{\alpha_0} \prod_{v\in T} \frac{(-1)^{\left| H_v\right|} \sigma^{\left| I_v\right|}(\alpha_{\left| H_v\right| + \left| I_v\right| - 1})}{\pi^{\lfloor (\left| H_v\right| + \left| I_v\right| - 1)/2\rfloor}},
\end{align}
where the subscript of $\alpha$ is understood as being defined modulo $2$.
\end{definition}

As in \cref{conventions section}, in Theorem \ref{main action thm} $A$ denotes the ring of integers in any finite field extension of $\mathbb{Q}_p$, $q$ denotes the cardinality of the residue field of $A$, $\pi$ denotes a fixed choice of uniformizer for $A$, and $\mathbb{G}$ denotes the Honda height $2$ formal $A$-module over $\mathbb{F}_{q^2}$, as defined at the start of \cref{Review of Devinatz...}. 
\begin{theorem}\label{main action thm}
Let $\alpha_0,\alpha_1\in W(\mathbb{F}_{q^2})$, with $\alpha_0$ a unit in $W(\mathbb{F}_{q^2})$. Then \[ \alpha_0 + \alpha_1S \in  \left(\mathcal{O}_F\otimes_{W(\mathbb{F}_q)}W(\mathbb{F}_{q^2})\langle S \rangle/\left(S^2 = \pi,\overline{\omega} S = S\omega \ \ \forall \omega\in W(\mathbb{F}_{q^2})\right)\right)^{\times} \cong \Aut(\mathbb{G})\] acts on the Hazewinkel coordinate $u_1$ for the Lubin-Tate ring $W(\mathbb{F}_{q^2})[[u_1]]$ as follows:
\begin{align}
\nonumber (\alpha_0 + \alpha_1S).u_1 &= \sum_{T\in qLT} \ind_{\alpha_0,\alpha_1}(T)\cdot u_1^{\wt(T)} \\
\label{sum 003941} &= \sum_{n\geq 1} \sum_{\substack{T\in qLT,\\ \wt(T) = n}} \ind_{\alpha_0,\alpha_1}(T)\cdot u_1^n.
\end{align}
\end{theorem}
\begin{proof}
Let $\Gamma_n$ denote the coefficient of $u_1^n$ in the right-hand side of \eqref{sum 003941}, i.e., $\Gamma_n$ is the sum of the $(\alpha_0,\alpha_1)$-indices of all ordered rooted trees $q$-labelled trees of weight $n$. Let $\gamma_n$ denote the coefficient of $u_1^n$ in the power series $(\alpha_0 + \alpha_1S).u_1$, i.e., the symbol $\gamma_n$ has the same meaning that it has everywhere else in this paper (e.g. in Proposition \ref{u1 formula}). We need to show that $\gamma_n = \Gamma_n$ for all $n$. Both $\gamma_n$ and $\Gamma_n$ are zero for $n\leq 0$. In the case $n=1$, it is straightforward to verify that $\Gamma_1 = \sigma(\alpha_0)/\alpha_0$, since for any prime power $q$, the only $q$-labelled ordered rooted tree of weight $1$ is a tree with a single vertex, with label $(H = (), I = (0))$. In the examples following Proposition \ref{u1 formula}, we saw that $\gamma_1 = \overline{\alpha}_0/\alpha_0$. Hence $\Gamma_1 = \gamma_1$. 

From here we will proceed by induction. Suppose $T$ is a $q$-labelled ordered rooted tree $T$ of weight $n$. Write $r$ for the root of $T$. Let $T_1, \dots ,T_m$ be the disjoint nonempty subgraphs of $T\backslash \{r\}$, as in Definition \ref{def of rooted tree}. The given $q$-labelling on $T$ restricts to a $q$-labelling on each of the subtrees $T_1, \dots ,T_m$. The $q$-labelled ordered tree $T$ is furthermore uniquely determined by specifying the ordered $m$-tuple of $q$-labelled ordered rooted trees $(T_1, \dots ,T_m)$, together with the label $(H_r, I_r)$ on the root $r$. The only constraints on the label $(H_r,I_r)$ are:
\begin{itemize}
\item $H_r$ and $I_r$ must be in $\J$, i.e., they must be monotone strictly increasing, parity-alternating sequences of nonnegative integers which do not begin with an odd number,
\item $QH_r = m$, and
\item the weight $n=\wt(r)$ is strictly greater than the sum of the weights of each of the subtrees $T_1, \dots ,T_m$.
\end{itemize}
On the $(\alpha_0,\alpha_1)$-indices, \eqref{index def} yields the relationship
\begin{align}
\label{index eq 1} \alpha_0 \cdot \ind_{\alpha_0,\alpha_1}(T) 
  &= \frac{(-1)^{\left| H_r\right|} \sigma^{\left| I_r\right|}(\alpha_{\left| H_r\right| + \left| I_r\right| - 1})}{\pi^{\lfloor (\left| H_r\right| + \left| I_r\right| - 1)/2\rfloor}} \cdot \prod_{i=1}^m \alpha_0\cdot \ind_{\alpha_0,\alpha_1}(T_i) 
\end{align}
between the indices of $T$ and of $T_1, \dots ,T_m$.

Our point is that there is bijection between, on one hand, the set of $q$-labelled ordered rooted trees of weight $n$, and on the other hand, the set of ordered triples $(H_r,I_r,K)$, where $K$ is a finite-length ordered sequence of $q$-labelled ordered rooted trees of weight $<n$, and $H_r,I_r\in \J$ satisfy the conditions:
\begin{itemize}
\item $QH_r = \left| K \right|$, 
\item and the sum of the weights of the $q$-labelled trees in $K$ is equal to $n - QI_r$.
\end{itemize}
From this bijection, and from its predictable effect \eqref{index eq 1} on the $(\alpha_0,\alpha_1)$-indices, we have that the sum $\Gamma_n$ of the $(\alpha_0,\alpha_1)$-indices of the $q$-labelled ordered rooted trees of weight $n$ satisfies the following recursion:
\begin{align}
\label{gamman recursion 5}\Gamma_n &= \frac{1}{a_0} 
\sum_{\substack{\text{sequences $K$ of}\\ \text{positive integers,}\\ K\neq (n)}}
\left( \prod_{k\in K} \Gamma_k\right) \sum_{\substack{H,I\in \J : \\ QH  = \left| K \right| ,\\ QI = n - \Sigma K }} (-1)^{\left| H\right|} \frac{\sigma^{\left| I\right|} (a_{\left|H\right| + \left|I\right| - 1})}{\pi^{\lfloor (\left|H\right| + \left|I\right|-1)/2\rfloor}}.
\end{align}
Recursion \eqref{gamman recursion 5} together with the value of $\Gamma_1$ suffice to determine all the higher values $\Gamma_2,\Gamma_3,\dots$. But recursion \eqref{gamman recursion 5} is the same as the recursion \eqref{gamman formula} which determines $\gamma_2,\gamma_3,\dots$ from knowledge of $\gamma_1$. Hence $\Gamma_n = \gamma_n$ for all $n$.\end{proof}

\begin{remark}
Let $q = p^{\mathfrak{f}}$. In the limit as $\mathfrak{f} \rightarrow \infty$, formula \eqref{sum 003941} simplifies to 
\begin{align}
\label{stable formula}  (\alpha_0 + \alpha_1S).u_1
   &= \frac{\overline{\alpha}_0 u_1}{\alpha_0 \cdot (1 + \overline{\alpha}_1 u_1)}.
\end{align}
Formula \eqref{stable formula} also arises from taking the limit of \eqref{sum 003941} as $p\rightarrow \infty$. At least for now, this simplification is only a curiosity, as the author does not know of any algebraic or topological application for \eqref{stable formula}.
\end{remark}

\section{The action of $W(\mathbb{F}_{q^2})^\times \subseteq \Aut(\mathbb{G})$ on $u_1$}

There is an evident copy of the unit group $W(\mathbb{F}_{q^2})^\times$ of the Witt ring inside\linebreak $\left(\mathcal{O}_F\otimes_{W(\mathbb{F}_q)} W(\mathbb{F}_{q^2})\langle S \rangle/\left(S^2 = \pi,\overline{\omega} S = S\omega \ \ \forall \omega\in W(\mathbb{F}_{q^2})\right)\right)^{\times} \cong \Aut(\mathbb{G})$. If $F/\mathbb{Q}_p$ is unramified, then this unit group is a maximal abelian closed subgroup of $\Aut(\mathbb{G})$. If we are content to describe the action of $W(\mathbb{F}_{q^2})^{\times}$ on the Lubin-Tate ring $\Def(\mathbb{G})$, rather than describing the action of the entirety of $\Aut(\mathbb{G})$, then some simplifications are possible. This tendency was observed already by Kohlhaase in \cite[Theorem 1.19]{MR3069363}, who gave an approximate description (i.e., a description modulo $(\pi,u_1)^{2q+4}$) of the action of $W(\mathbb{F}_{q^2})^\times$ on $\Def(\mathbb{G})$. In this section we will give an {\em exact} description of that group action on $\Def(\mathbb{G})$, as a special case of Proposition \ref{u1 formula}, under the assumption that $q = p^\mathfrak{f}$ for {\em odd} $\mathfrak{f}$. (This of course includes the most important case, the base case $\mathfrak{f}=1$.)

\begin{lemma}\label{degree concen lemma}
Let $q = p^\mathfrak{f}$ for an odd integer $\mathfrak{f}$.
Given an element $\alpha\in W(\mathbb{F}_{q^2})^\times\subseteq\Aut(\mathbb{G})$,
let $\gamma_0,\gamma_1,\gamma_2,\dots\in W(\mathbb{F}_{q^2})$ be as in Proposition \ref{u1 formula}, i.e., $\alpha  . u_1= \sum_{n\geq 1} \gamma_n u_1^n$. Then $\gamma_n=0$ for all $n\nequiv 1\mod p+1$.
\end{lemma}
\begin{proof}
By induction. From Proposition \ref{u1 formula} it is immediate that $\gamma_0=0$. Suppose $n$ is an integer, $n\nequiv 1\mod p+1$, and suppose we have already shown that $\gamma_k=0$ for all $k<n$ such that $k\nequiv 1 \mod p+1$.
Recall, from Definition \ref{def of J}, that $\J^{even}$ (respectively, $\J^{odd}$) is the set of members of $\J$ which are sequences of even length (respectively, odd length).
The special case $\alpha_0=\alpha$ and $\alpha_1=0$ of equation \eqref{gamman formula} from Proposition \ref{u1 formula} takes the form
\begin{align}\label{gamman formula 2}
\gamma_n &= 
\sum_{\substack{\text{sequences $K$ of}\\ \text{positive integers,}\\ K\neq (n)}}
\left( \prod_{k\in K} \gamma_k\right) \left( \sum_{\substack{H\in \J^{even},\\ I\in \J^{odd} : \\ QH  = \left| K \right| ,\\ QI = n - \Sigma K }}  \frac{\overline{\alpha}/\alpha}{\pi^{(\left|H\right| + \left|I\right| -1)/2}}
 - \sum_{\substack{H\in \J^{odd},\\ I\in \J^{even} : \\ QH  = \left| K \right| ,\\ QI = n - \Sigma K }}  \frac{1}{\pi^{(\left|H\right| + \left|I\right| -1)/2}}\right).
\end{align}
Suppose that $K$ is a sequence of integers, and $H,I$ are members of $\J$, such that the triple $K,H,I$ yields a nonzero summand in the right-hand side of \eqref{gamman formula 2}. We must prove that $n\equiv 1\mod p+1$.
By the inductive hypothesis, $K$ can only consist of integers congruent to $1$ modulo $p+1$, hence $\Sigma K \equiv \left| K \right|$ modulo $p+1$. Since $q=p^\mathfrak{f}$ for an {\em odd} integer $\mathfrak{f}$, we have that $q\equiv -1$ modulo $p+1$. Since $I$ is parity-alternating, this implies that $QI \equiv 0$ modulo $p+1$ if $\left| I\right|$ is even, while $QI\equiv 1$ modulo $p+1$ if $\left| I\right|$ is odd. The same claim holds with $H$ in place of $I$ throughout. We furthermore know that $\left| H\right|$ and $\left| I\right|$ have opposite parity, i.e., $QH+QI\equiv 1$ modulo $p+1$. Hence we have congruences
\begin{align*}
 1 
  &\equiv QH + QI \\
  &\equiv \left| K \right| + n - \Sigma K \\
  &\equiv n \mod p+1,
 \end{align*}
as desired.
\end{proof}

\begin{definition}\label{def of q-alt trees}
Given a prime power $q$ and an ordered rooted tree $T$, we will say that a $q$-labelling of $T$ is {\em alternating} if, for each vertex $v$ of $T$, the parity of $\left| H_v\right|$ is opposite of the parity of $\left| I_v\right|$.
As shorthand for the phrase ``ordered rooted tree with an alternating $q$-labelling,'' we will simply write ``$q$-alternating tree.'' We will write $qA$ for the set of all finite $q$-alternating trees. 

Given an $q$-alternating tree $T$  and a nonzero element $\alpha$ in an integral domain $R$ equipped with an involution $\sigma$, the {\em $\alpha$-index} of $T$ is defined as the product:
\begin{align*}
\ind_{\alpha}(T) &= \frac{1}{\alpha} \prod_{v\in T} \frac{(-1)^{\left| H_v\right|} \sigma^{\left| I_v\right|}(\alpha)}{\pi^{\left(\left| H_v\right| + \left| I_v\right|-1\right)/2}}.
\end{align*}
\end{definition}

\begin{theorem}\label{witt u1 formula}
If $q = p^\mathfrak{f}$ for $\mathfrak{f}$ odd, then any element $\alpha\in W(\mathbb{F}_{q^2})^\times\subseteq\Aut(\mathbb{G})$ acts on $\Def(\mathbb{G})\cong W(\mathbb{F}_{q^2})[[u_1]]$ via the following formula: 
\begin{align}
\label{deltan formula} \alpha.u_1 &=  \sum_{n\geq 1} \sum_{\substack{\text{$T\in qA$}\\ \wt(T)=1+(p+1)n}} \ind_{\alpha}(T) \cdot u_1^{1 + (p+1)n}.
\end{align}
\end{theorem}
\begin{proof}
Apply Lemma \ref{degree concen lemma} to formula \eqref{sum 003941} and simplify using elementary algebra.
\end{proof}

\begin{corollary}\label{witt comp cor 1}
Let $q = p$, and let $\alpha\in W(\mathbb{F}_{q^2})^\times\subseteq\Aut(\mathbb{G})$. Write $\beta$ as shorthand for the quotient $\overline{\alpha}/\alpha$. Then $\alpha$ acts on $u_1$ as follows:
\begin{align}
\nonumber \alpha . u_1 
  &\equiv 
   \beta u_1
  +\frac{1}{\pi}\beta \left( \beta^{p+1}-1\right) u_1^{p+2} 
  + \frac{1}{\pi^2}\beta \left( \beta^{p+1}-1\right)\left( (p+1)\beta^{p+1} - 1\right) u_1^{2p+3} \\
\label{witt action 11} & \hspace{29pt}  + \frac{1}{\pi^3} \beta  \left( \beta^{p+1}-1\right) \left( \left((p+1)\beta^{p+1}-1\right)^2 + \binom{p+1}{2} \beta^{p+1} (\beta^{p+1}-1)\right)u_1^{3p+4} \\
\nonumber &\hspace{260pt}\mod (u_1^{4p+5},u_1^{p^2}).\end{align}
\end{corollary}
\begin{proof}
In the case $q=p$, the only $q$-alternating trees of weight $\leq \min\{ 4p+4,p^2-1\}$ are as follows. We will draw the labelled trees using the same graphical conventions explained and used in \cref{Drawn examples...}.
\begin{description}
\item[Weight $1$] The only weight $1$ $p$-labelled ordered rooted tree is depicted in \cref{Drawn examples...}, and it is $p$-alternating. For ease of comparison with the higher-weight $p$-alternating trees, we reproduce that depiction here:
$\boxed{\parbox{50pt}{weight $1$\\index $\frac{\overline{\alpha}}{\alpha}$} \vcenter{\xymatrix{
*+[F]\txt{()\\(0)}}
}}$.
The total $\alpha$-index is consequently $\overline{\alpha}/\alpha$, yielding the coefficient of $u_1$ in \eqref{witt action 11}.
\item[Weight $p+2$] Suppose that $p+2<p^2$. Then there are two $p$-alternating trees of weight $p+2$:
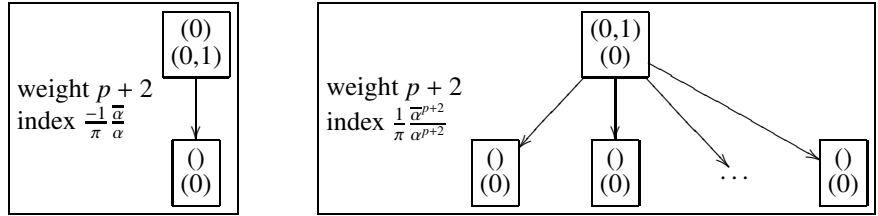
\[
\boxed{\parbox{55pt}{weight $p+2$\\index $\frac{-1}{\pi}\frac{\overline{\alpha}}{\alpha}$} \vcenter{\xymatrix{ *+[F]\txt{(0)\\(0,1)} \ar[d] \\
*+[F]\txt{()\\(0)}}
}}\hspace{30pt}
\boxed{\parbox{55pt}{weight $p+2$\\index $\frac{1}{\pi}\frac{\overline{\alpha}^{p+2}}{\alpha^{p+2}}$} \vcenter{\xymatrix{
 & *+[F]\txt{(0,1)\\(0)} \ar[ld] \ar[d]\ar[rd] \ar[rrd] & & \\
*+[F]\txt{()\\(0)} & *+[F]\txt{()\\(0)} & \dots & *+[F]\txt{()\\(0)}
}}}\]
where, in the right-hand $p$-alternating tree, the bottom row consists of $p+1$ copies of the weight $1$ $p$-alternating tree.

The sum of the $\alpha$-indices of the weight $p+2$ $p$-alternating trees, i.e., the coefficient of $u_1^{p+2}$ in \eqref{witt action 11}, is consequently $\frac{1}{\pi}\frac{\overline{\alpha}}{\alpha}\left( \left( \frac{\overline{\alpha}}{\alpha}\right)^{p+1} - 1\right)$.
\item[Weight $2p+3$] Suppose that $2p+3<p^2$. Then there are a total of $2p+4$ $p$-alternating trees of weight $2p+3$. They come in two families, arising from two different operations that take as input a $p$-alternating tree of weight $p+2$, and produce as output a $p$-alternating tree of weight $2p+3$. Here are their graphical depictions:
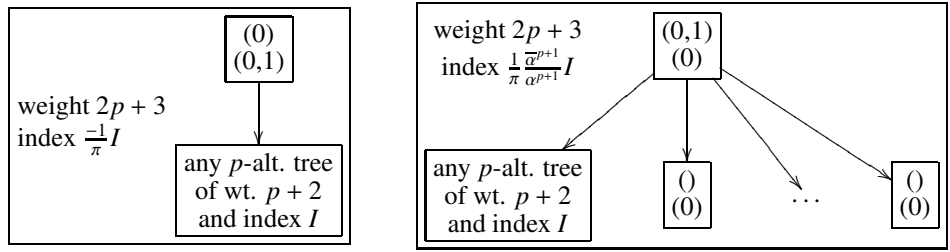
\[
\boxed{\parbox{60pt}{weight $2p+3$\\index $\frac{-1}{\pi}I$} \vcenter{\xymatrix{
 *+[F]\txt{(0)\\(0,1)} \ar[d] \\
*+[F]\txt{any $p$-alt. tree\\of wt. $p+2$\\and index $I$}
}}}
\hspace{25pt}
\boxed{\vcenter{\xymatrix{
 \txt{weight $2p+3$\\index $\frac{1}{\pi}\frac{\overline{\alpha}^{p+1}}{\alpha^{p+1}}I$} & *+[F]\txt{(0,1)\\(0)} \ar[ld] \ar[d]\ar[rd] \ar[rrd] & & \\
*+[F]\txt{any $p$-alt. tree\\of wt. $p+2$\\and index $I$} & *+[F]\txt{()\\(0)} & \dots & *+[F]\txt{()\\(0)}
}}}
\]
Since there are two $p$-alternating trees of weight $p+2$, the left-hand diagram produces two $p$-alternating trees of weight $2p+3$. However, in the right-hand diagram, we are supposed to understand that there are $p$ copies of the weight $1$ tree in the bottom row, and any weight $p+2$ tree can be inserted in the drawn slot, {\em and we also have the $p+1$ permutations of the right-hand diagram given by permuting the entire bottom row.} Consequently there are $2(p+1)$ $p$-alternating trees of weight $2p+3$ which are produced by applying the $p+1$ permutations of the right-hand diagram to each of the two $p$-alternating diagrams of weight $p+1$.

Hence the sum of the $\alpha$-indices of the weight $2p+3$ $p$-alternating trees, i.e., the coefficient of $u_1^{2p+3}$ in \eqref{witt action 11}, is $\frac{1}{\pi}\left( (p+1)(\frac{\overline{\alpha}}{\alpha})^{p+1} - 1\right)$ times the coefficient of $u_1^{p+2}$ in \eqref{witt action 11}.
\item[Weight $3p+4$] Suppose that $3p+4<p^2$. Then each $p$-alternating tree of weight $3p+4$ arises from one of three operations that takes as input $p$-alternating trees of lower weight. The first two such operations are as depicted below:
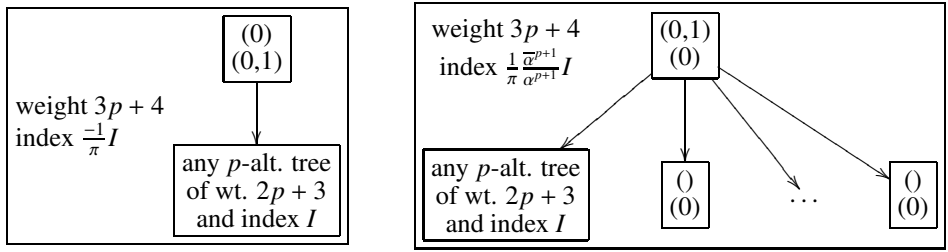
\[
\boxed{\parbox{60pt}{weight $3p+4$\\index $\frac{-1}{\pi}I$} \vcenter{\xymatrix{
 *+[F]\txt{(0)\\(0,1)} \ar[d] \\
*+[F]\txt{any $p$-alt. tree\\of wt. $2p+3$\\and index $I$}
}}}
\hspace{25pt}
\boxed{\vcenter{\xymatrix{
 \txt{weight $3p+4$\\index $\frac{1}{\pi}\frac{\overline{\alpha}^{p+1}}{\alpha^{p+1}}I$} & *+[F]\txt{(0,1)\\(0)} \ar[ld] \ar[d]\ar[rd] \ar[rrd] & & \\
*+[F]\txt{any $p$-alt. tree\\of wt. $2p+3$\\and index $I$} & *+[F]\txt{()\\(0)} & \dots & *+[F]\txt{()\\(0)}
}}}
\]
We also have the $p+1$ permutations of the bottom row of the diagram depicted on the right-hand side. Finally, the last such operation is:
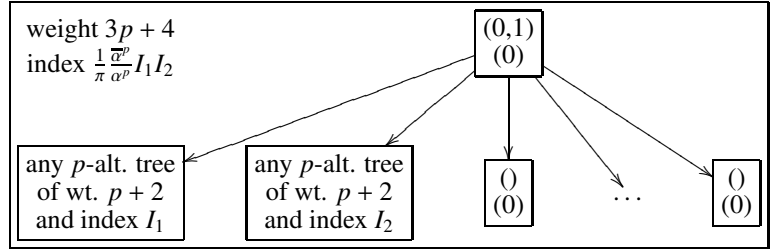
\[
\boxed{\vcenter{\xymatrix{
 \txt{weight $3p+4$\\index $\frac{1}{\pi}\frac{\overline{\alpha}^{p}}{\alpha^{p}}I_1I_2$} & & *+[F]\txt{(0,1)\\(0)} \ar[lld]\ar[ld] \ar[d]\ar[rd] \ar[rrd] & & \\
*+[F]\txt{any $p$-alt. tree\\of wt. $p+2$\\and index $I_1$} & *+[F]\txt{any $p$-alt. tree\\of wt. $p+2$\\and index $I_2$} & *+[F]\txt{()\\(0)} & \dots & *+[F]\txt{()\\(0)}
}}}
\]
along with the various permutations of the bottom row. 
The bottom row of this last diagram must have a total of $p+1$ vertices.
One must exercise a bit of care to correctly count the $p$-alternating trees arising from the last operation, since if $T,T^{\prime}$ are $p$-alternating trees of weight $p+2$, then the last operation yields a total of $p(p+1)$ $p$-alternating trees if $T\neq T^{\prime}$, but only $\binom{p+1}{2}$ $p$-alternating trees if $T = T^{\prime}$. 

Using these operations to carefully sum up the $\alpha$-indices of all $p$-alternating trees of weight $3p+4$, one gets
\begin{eqnarray*}
\frac{-1}{\pi} \delta_2 + \frac{p+1}{\pi} \left( \frac{\overline{\alpha}}{\alpha}\right)^{p+1} \delta_2- \frac{(p+1)p}{\pi^2}\beta^{p+3} + \binom{p+1}{2} \left( \frac{-\beta}{\pi}\right)^2 +\binom{p+1}{2} \left( \frac{\beta^{p+2}}{\pi}\right)^2 ,
\end{eqnarray*}
where $\delta_2$ denotes the sum of the $\alpha$-indices of all $p$-alternating trees of weight $2p+3$. The claimed coefficient of $u_1^{3p+4}$ follows.
\end{description}
\end{proof}
Corollary \ref{witt comp cor 1} extends the calculations of Kohlhaase \cite[Theorem 1.19]{MR3069363}, which makes the same calculation in greater generality (assuming only $q>3$, not $q=p$), but to significantly lower accuracy: Kohlhaase's calculation is carried out only modulo $(\pi,u_1^{2q+4})$. 

\begin{remark}
The classification of low-weight $p$-alternating diagrams in the proof of Corollary \ref{witt comp cor 1} resembles the kinds of diagrams drawn when describing operations of various arities in some operad. It is natural to wonder whether there is any way to organize the labelled trees of the types that arise in this paper into an operad of some kind. The author does not plan to pursue that question, but would be glad to hear from anybody who does.
\end{remark}

Devinatz and Hopkins show in \cite{MR1333942} that $\Aut(\mathbb{G})$ acts {\em linearly} on the Cartier coordinate $w_1$ for $\mathbb{Q}\otimes_{\mathbb{Z}}W(\mathbb{F}_{p^2})[[u_1]]$, that is, for any $\alpha\in\Aut(\mathbb{G})$, $\alpha.w_1$ is equal to a scalar in $W(\mathbb{F}_{p^2})$ times $w_1$. The Hazewinkel coordinate $u_1$ is more natural in various ways than the Cartier coordinate $w_1$, e.g. the fact that $u_1$ is already an element of the Lubin-Tate ring and does not require passing to its divided power envelope. 
One would like to know whether there are nontrivial subgroups $H$ of $\Aut(\mathbb{G})$ with the property that $H$ acts linearly on the {\em Hazewinkel} coordinate $u_1$. Of course the center $\hat{\mathbb{Z}}_p^{\times}$ of $\Aut(\mathbb{G})$ acts linearly on $u_1$, since it acts trivially on $u_1$. It is also relatively well-known (e.g. by the argument about roots of unity in \cite[section 2.1.2]{MR782555}) that any $(p^2-1)$st root of unity in $W(\mathbb{F}_{p^2})$ also acts linearly on $u_1$. 

Hence the subgroup $\hat{\mathbb{Z}}_p^{\times}\times C_{p+1}$ of $W(\mathbb{F}_{p^2})^{\times}\subseteq \Aut(\mathbb{G})$ acts linearly on $u_1$. One simple consequence of Theorems \ref{main action thm} and \ref{witt u1 formula} is that this subgroup is {\em all} that acts linearly on $u_1$:
\begin{corollary}
Let $F=\mathbb{Q}_p$, i.e., $\pi = p = q$. Let $H\subseteq \Aut(\mathbb{G})$ act linearly on $u_1$. Then $H$ is a subgroup of $\hat{\mathbb{Z}}_p^{\times}\times C_{p+1} \subseteq W(\mathbb{F}_{p^2})^{\times}\subseteq \Aut(\mathbb{G})$.
\end{corollary}
\begin{proof}
From Theorem \ref{main action thm}, we see that $(\alpha_0+\alpha_1S).u_1 \equiv \frac{\overline{\alpha}_0}{\alpha_0} u_1 + \frac{-\overline{\alpha}_0\cdot\overline{\alpha_1}}{\alpha_0^2} u_1^2\mod u_1^3$. Since $\alpha_0$ cannot be zero, the only way for $\alpha_0+\alpha_1S$ to act linearly on $u_1$ is for $\alpha_1$ to be zero, i.e., $\alpha_0+\alpha_1S = \alpha_0\subseteq W(\mathbb{F}_{p^2})^{\times}$. From \eqref{witt action 11}, we see that $\alpha_0.u_1 \equiv \frac{\overline{\alpha}}{\alpha}u_1 +\frac{\overline{\alpha}}{\alpha}\frac{(\overline{\alpha}/\alpha)^{p+1}-1}{\pi} u_1^{p+2}$ modulo $u_1^{2p+3}$. Hence, if $\alpha$ acts linearly on $u_1$, we must have that $(\overline{\alpha}/\alpha)^{p+1} = 1$. 

Write $T$ for the group homomorphism $T: W(\mathbb{F}_{p^2})^{\times}\rightarrow W(\mathbb{F}_{p^2})^{\times}$ given by $T(x) = \overline{x}/x$. We have the exact sequence of profinite abelian groups
\[ 0 \rightarrow \hat{\mathbb{Z}}_p^{\times} \rightarrow W(\mathbb{F}_{p^2})^{\times} \stackrel{T}{\longrightarrow} W(\mathbb{F}_{p^2})^{\times} \rightarrow H_0(\Gal(\mathbb{F}_{p^2}/\mathbb{F}_p); W(\mathbb{F}_{p^2})^{\times}) \rightarrow 0,\]
and the group of elements $\alpha\in W(\mathbb{F}_{p^2})^{\times}$ satisfying $(\overline{\alpha}/\alpha)^{p+1} = 1$ is precisely the preimage of $C_{p+1}\subseteq W(\mathbb{F}_{p^2})^{\times}$ under the map $T$. 

Hence $H$ is contained in that preimage, i.e., $H$ is contained in the pullback $W(\mathbb{F}_{p^2})^{\times} \times_{\im T} C_{p+1} \cong \hat{\mathbb{Z}}_p^{\times} \times C_{p+1}$. 
\end{proof}

We give one more result on the action of $W(\mathbb{F}_{p^2})^{\times}$ on $u_1$: Theorem \ref{witt comp cor 2}, which describes the action of $W(\mathbb{F}_{p^2})^{\times}$ on $u_1$ modulo $u_1^{p^2+1}$. This is a substantially higher accuracy than Corollary \ref{witt comp cor 1}, and higher than any similar formula to be found in the literature (but of course it is not as accurate as Theorem \ref{witt u1 formula}, which describes the action exactly). Theorem \ref{witt comp cor 2} is, however, only a recursion, not a closed formula. 
It has the merit of avoiding any talk of trees, since the the formula is given in terms of weak compositions. In combinatorics, a {\em composition} (respectively, {\em weak composition}) of an integer $n$ is a finite sequence $(a_1,\dots ,a_{\ell})$ of positive integers (respectively, nonnegative integers) whose sum is $n$. The {\em length} of the composition $(a_1,\dots ,a_{\ell})$ is $\ell$. 
\begin{theorem}\label{witt comp cor 2}
  If $q = p$, then for any element $\alpha\in W(\mathbb{F}_{q^2})^\times\subseteq\Aut(\mathbb{G})$, we have $\alpha.u_1 \equiv \sum_{m\geq 0} \delta_0 u_1^{1+(p+1)m}$ modulo $u_1^{p^2+1}$, where the sequence $\delta_0, \delta_1,\delta_2,\dots \in W(\mathbb{F}_{q^2})$ is given by the recursion 
\begin{align}
\label{deltan formula 2} \delta_m & = -\frac{\delta_{m-1}}{\pi} + \sum_{\substack{\text{length\ } p+1\text{\ weak}\\\text{compositions\ } K\\ \text{of\ } m-1}} \frac{\overline{\alpha}/\alpha}{\pi}\prod_{k\in K} \delta_k \text{\ \ \ if\ } 1\leq m\leq p-2,\\
\label{deltan formula 3} \delta_{p-1} & = \frac{\overline{\alpha}}{\alpha} -  \left( \frac{\overline{\alpha}}{\alpha}\right)^{p^2}-\frac{\delta_{p-2}}{\pi} + \sum_{\substack{\text{length\ } p+1\text{\ weak}\\\text{compositions\ } K\\ \text{of\ } p-2}} \frac{\overline{\alpha}/\alpha}{\pi}\prod_{k\in K} \delta_k.
\end{align}
\end{theorem}
\begin{proof}
From Lemma \ref{degree concen lemma}, the only nonzero coefficients of $\alpha.u_1$ modulo $u_1^{p^2+1}$ occur in degrees of the form $u_1^{1+(p+1)m}$ with $0\leq m\leq p-1$. From Proposition \ref{u1 formula}, the only contributions in those degrees come from $4$-tuples $(m,H,I,K)$, with $K$ a finite sequence of nonnegative integers, and with $H,I\in \J$ of lengths of opposite parity of one another and satisfying $QH = \left| K \right|$ and $QI = 1 - \left| K \right| + (p+1)(m-\Sigma K)$. 
The only such $4$-tuples are:
\begin{itemize}
\item $H=(0),I=(),K=(m),m\in \{1,\dots ,p-1\}$, the case which is excluded from the sum \eqref{gamman formula},
\item $H = (), I=(0),K=(),m=0$, which contributes only the summand $\frac{\overline{\alpha}}{\alpha}$ to $\delta_0$, 
\item $H=(0,1), I=(0),K$ any weak composition of $m-1$ of length $p+1$, $m\in \{1,\dots ,p-1\}$, which contributes the summand \[\sum_{\substack{\text{length\ } p+1\text{\ weak}\\\text{compositions\ } K\\ \text{of\ } m-1}} \frac{\overline{\alpha}/\alpha}{\pi}\prod_{k\in K} \delta_k\]  to $\delta_m$,
\item $H=(0), I=(0,1),K = (m-1),m\in \{1,\dots ,p-1\}$, which contributes the summand $\frac{-\delta_{m-1}}{\pi}$ to $\delta_m$,
\item $H = (2),I = (),$ and $K=(0,...,0)$ of length $p^2$, and $m = p-1$, which contributes the summand $-\delta_0^{p^2}$ to $\delta_{p-1}$,
\item and $H = (), I=(2),K=(),m=p-1$, which contributes the summand $\overline{\alpha}/\alpha$ to $\delta_{p-1}$.
\end{itemize}
The recursion \eqref{gamman formula} in these degrees consequently becomes \eqref{deltan formula 2} and \eqref{deltan formula 3}.
\end{proof}

\section{The action of $\Aut(\mathbb{G})$ and of $W(\mathbb{F}_{q^2})$ on the Bott element $u$}

With the closed formula for the Devinatz--Hopkins series $f_1(u_1)$ and $f(u_1)$ provided by Theorem \ref{w1 calc}, and the closed formula for $(\alpha_0+\alpha_1S).u_1$ provided by Theorem \ref{main action thm}, it is in principle straightforward to describe the action of $\Aut(\mathbb{G})$ on the graded deformation ring $W(\mathbb{F}_{q^2})[[u_1]][u^{\pm 1}]$, and not merely on its degree zero subring $W(\mathbb{F}_{q^2})[[u_1]]$. The idea is simple: since $u = w/f(u_1)$ and $ww_1 = f_1(u_1)\cdot u$, and since the Cartier theory gives us that $(\alpha_0 + \alpha_1S).w = \overline{\alpha}_1ww_1 + \alpha_0w$, we have equalities
\begin{align}
\nonumber 
 (\alpha_0 + \alpha_1S).u
  &= \frac{(\alpha_0 + \alpha_1S).w}{(\alpha_0 + \alpha_1S).f} \\
\label{power series eq 1}  &= \frac{\overline{\alpha}_1 f_1(u_1) + \alpha_0 f(u_1)}{f\left((\alpha_0 + \alpha_1S).u_1\right)}\cdot u .
\end{align}
To be clear, the expression $f\left((\alpha_0 + \alpha_1S).u_1\right)$ is the composite of the power series $f(u_1)$ with the power series $(\alpha_0 + \alpha_1S).u_1$ calculated in Theorem \ref{main action thm}.
Each of the power series in \eqref{power series eq 1} has been computed already in this paper, so in some sense we are done.


However, formula \eqref{power series eq 1} is not nearly as explicit as Theorem \ref{main action thm}. It would be better to write $(\alpha_0 + \alpha_1S).u$ as a power series $(\theta_0 + \theta_1 u_1 + \theta_2 u_1^2 + \dots )\cdot u$ and to give a closed formula for the coefficients $\theta_0,\theta_1,\dots \in W(\mathbb{F}_{q^2})$. This is the purpose of this section, whose main result is Theorem \ref{main action thm on u}.

We need to introduce two more pieces of notation before stating the theorem. Given an $n$-tuple $J = (T_1, \dots ,T_n)$ of elements of $qLT$, we write $\ind_{\alpha_0,\alpha_1}(J)$ as shorthand for the {\em product} of the indices of the members of $J$, i.e., $\ind_{\alpha_0,\alpha_1}(J) = \prod_{i=1}^n \ind_{\alpha_0,\alpha_1}(T_i)$. Similarly, we write $\wt(J)$ as shorthand for the {\em sum} of the weights of the members of $J$, i.e., $\wt(J) = \sum_{i=1}^n \wt(T_i)$.
\begin{theorem}\label{main action thm on u}
Let $A,q,\pi,\mathbb{G},\mathfrak{f}$ be as in \cref{conventions section}. Assume the residue degree $\mathfrak{f}$ is odd. Let $\alpha_0,\alpha_1\in W(\mathbb{F}_{q^2})$, with $\alpha_0$ a unit in $W(\mathbb{F}_{q^2})$. Then $\alpha_0 + \alpha_1S \in  
\Aut(\mathbb{G})$ acts on the coordinate $u$ for the graded Lubin-Tate ring $W(\mathbb{F}_{q^2})[[u_1]][u^{\pm 1}]$ by the formula $(\alpha_0 + \alpha_1S).u = \sum_{n\geq 1} \theta_n\cdot u_1^n\cdot u$, where the coefficients $\theta_0,\theta_1,\theta_2,\dots \in W(\mathbb{F}_{q^2})$ are given by
\begin{align}
\label{action on u closed formula} \theta_n &=  \sum_{i=0}^{n}\left( \left(\overline{\alpha}_1 \sum_{\substack{I\in \J^{odd}\\ QI=i}} \frac{1}{\pi^{\left(\left| I \right| - 1\right)/2}} + \alpha_0 \sum_{\substack{I\in \J^{even}\\ QI=i}} \frac{1}{\pi^{\left| I \right|/2}}\right) \cdot \sum_{\substack{\text{sequences}\\ \text{$K$ of positive}\\ \text{integers,}\\ \text{ $\Sigma K = n-i$}}} (-1)^{\left| K\right|} \prod_{k\in K} \left(\sum_{\substack{\text{$I\in \J^{even}$}\\\text{$J\in qLT^{QI}$}\\ \text{$\wt(J) = k$}}} \frac{\ind_{\alpha_0,\alpha_1}(J)}{\pi^{\left| I\right|/2}}\right)
\right).
\end{align} 
\end{theorem}
\begin{proof}
From the equations
\begin{align*}
 \sum_{n\geq 0} \left( \overline{\alpha}_1 \sum_{\substack{I\in \J^{odd} \\ QI=n}} \frac{1}{\pi^{\left( \left| I\right| - 1\right)/2}} + \alpha_0 \sum_{\substack{I\in \J^{even} \\ QI=n}} \frac{1}{\pi^{\left| I\right|/2}} \right)\cdot u_1^n \cdot u
  &= \overline{\alpha}_1 f_1\cdot u + \alpha_0 f\cdot u \\
  &= \left((\alpha_0 + \alpha_1S).u\right) \cdot f((\alpha_0 + \alpha_1S).u_1) \\
  &= \left( \sum_{n\geq 0} \theta_n u_1^n\right)\cdot u \cdot \sum_{I\in \J^{even}} \frac{ ((\alpha_0 + \alpha_1S).u_1)^{QI}}{\pi^{\left| I\right|/2}} \\
  &= \left( \sum_{n\geq 0} \theta_n u_1^n\right) \cdot \sum_{I\in \J^{even}} \frac{1}{\pi^{\left| I\right|/2}} \sum_{J\in qLT^{QI}} \ind_{\alpha_0,\alpha_1}(J)\cdot u_1^{\wt(J)} \cdot u \\
  &= \sum_{n\geq 0}  \sum_{I\in \J^{even}} \frac{1}{\pi^{\left| I\right|/2}} \sum_{\substack{J\in qLT^{QI}\\ \wt(J) \leq n}} \ind_{\alpha_0,\alpha_1}(J) \cdot \theta_{n-\wt(J)} \cdot u_1^n\cdot u
\end{align*}
we get, for each $n\geq 1$, the recursion
\begin{align*}
 \theta_n &= \overline{\alpha}_1 \sum_{\substack{I\in \J^{odd}\\ QI=n}} \frac{1}{\pi^{\left(\left| I \right| - 1\right)/2}} + \alpha_0 \sum_{\substack{I\in \J^{even}\\ QI=n}} \frac{1}{\pi^{\left| I \right|/2}} - \sum_{i=0}^{n-1}\theta_i \sum_{I\in \J^{even}} \frac{1}{\pi^{\left| I\right|/2}} \sum_{\substack{J\in qLT^{QI}\\ 1\leq \wt(J) = n-i}} \ind_{\alpha_0,\alpha_1}(J) .
\end{align*}
It is an elementary matter to solve any recursion of the form $\theta_n = x_n - \sum_{i=0}^{n-1}\theta_i y_{n-i}$, for any sequences $x_0,x_1, \dots$ and $y_0,y_1,\dots$, to get
\begin{align}
\label{elementary eq 1}\theta_n &= \sum_{i=0}^{n} x_i \sum_{\substack{\text{sequences}\\ \text{$K$ of positive}\\ \text{integers,}\\ \text{$\Sigma K = n-i$}}} (-1)^{\left| K\right|} \prod_{k\in K} y_k.\end{align} 
In the case 
\begin{align*}
 x_n &= \overline{\alpha}_1 \sum_{\substack{I\in \J^{odd}\\ QI=n}} \frac{1}{\pi^{\left(\left| I \right| - 1\right)/2}} + \alpha_0 \sum_{\substack{I\in \J^{even}\\ QI=n}} \frac{1}{\pi^{\left| I \right|/2}}\mbox{\ \ \ and} \\ 
 y_n &= \sum_{I\in \J^{even}} \frac{1}{\pi^{\left| I\right|/2}} \sum_{\substack{J\in qLT^{QI}\\ \wt(J) = n}} \ind_{\alpha_0,\alpha_1}(J),
\end{align*}
equation \eqref{elementary eq 1} becomes \eqref{action on u closed formula}.
\end{proof}

Recall that $qA$ denotes the set of $q$-alternating trees, as in Definition \ref{def of q-alt trees}. Given a nonnegative integer $n$, we will write $qA^n$ for the set of ordered $n$-tuples of members of $qA$.
\begin{corollary}\label{witt action on u thm}
If $q = p^\mathfrak{f}$ for $\mathfrak{f}$ odd, then any element $\alpha\in W(\mathbb{F}_{q^2})^\times\subseteq\Aut(\mathbb{G})$ acts on $\Def(\mathbb{G})\cong W(\mathbb{F}_{q^2})[[u_1]]$ via the following formula: $\alpha . u = \sum_{n\geq 0} \tau_n \cdot u_1^{(p+1)n}\cdot u$, where $\tau_0,\tau_1,\dots \in W(\mathbb{F}_{q^2})$ is given by
\begin{align}
\label{tau formula 1}
 \tau_n &= \alpha \cdot \sum_{i=0}^{n}\left( \left(  \sum_{\substack{\text{$I\in \J^{even}$}\\ \text{$QI=n(p+1)$}}} \frac{1}{\pi^{\left| I\right|/2}}\right) \cdot \sum_{\substack{\text{sequences}\\ \text{$K$ of positive}\\ \text{integers,}\\ \text{$\Sigma K = n-i$}}} (-1)^{\left| K\right|} \prod_{k\in K} \sum_{\substack{\text{$I\in \J^{even}$}\\ \text{$J\in qA^{QI}$}\\ \text{$\wt(J)= k(p+1)$}}} \frac{\ind_{\alpha}(J)}{\pi^{\left| I\right|/2}}\right)
\end{align}
\end{corollary}
\begin{proof}
This proof is quite similar to that of Theorem \ref{main action thm on u}, but using Theorem \ref{witt u1 formula} rather than Theorem \ref{main action thm} to get the action of $\alpha$ on $u_1$. 
The methods of Devinatz--Gross--Hopkins give\footnote{We remind the reader that our notational conventions from \cref{conventions section} distinguish between $\alpha.w$ and $\alpha\cdot w$, and it is a special case of \cite[Proposition 3.3]{MR1333942} that $\alpha.w$ and $\alpha\cdot w$, despite being defined differently, are in fact equal.} equalities $w= f\cdot u$ and $\alpha.w = \alpha\cdot w$, and using Theorem \ref{w1 calc} to describe the power series $f$, we get equalities
\begin{align}
\label{eq a1} \alpha \cdot \sum_{I\in \J^{even}} \frac{u_1^{QI}}{\pi^{\left| I\right|/2}} \cdot u
  &= (\alpha \cdot f)\cdot u \\
\nonumber  &= \alpha \cdot w \\
\nonumber  &= \alpha.w \\
\nonumber  &= (\alpha.u)\cdot (\alpha.f) \\ 
\nonumber  &= \left( \sum_{n\geq 0} \tau_n u_1^{(p+1)n} \right) \cdot \left( \sum_{I\in \J^{even}} \frac{(\alpha.u_1)^{QI}}{\pi^{\left| I\right|/2}} \right) \cdot u \\
\label{eq 340k9f}  &= \left( \sum_{n\geq 0} \tau_n u_1^{(p+1)n} \right) \cdot \left( \sum_{I\in \J^{even}} \frac{1}{\pi^{\left| I\right|/2}} \sum_{J\in qA^{QI}} \ind_{\alpha}(J)\cdot u_1^{\wt(J)}\right) \cdot u \\
\label{eq a2}  &= \sum_{n\geq 0} \sum_{I\in \J^{even}} \left( \frac{1}{\pi^{\left| I\right|/2}} \sum_{\substack{\text{$J\in qA^{QI}$}\\ \text{$\wt(J)\leq n(p+1)$}}} \tau_{n-\frac{\wt(J)}{p+1}} \cdot \ind_{\alpha}(J) \right) \cdot u_1^{n(p+1)}\cdot u ,
\end{align}
where \eqref{eq 340k9f} follows from the following observations:
\begin{itemize}
\item if $I\in \J^{even}$, then $QI$ must be congruent to $0$ modulo $p+1$,
\item so, if $J$ is furthermore a $QI$-tuple of $q$-alternating trees, each of which has weight congruent to $1$ modulo $p+1$, then the weight of $J$ itself is a multiple of $p+1$.
\end{itemize}
From the equality of \eqref{eq a1} and \eqref{eq a2}, we get the recursion
\begin{align}
 \label{recursion ab2340}
 \tau_n
  &= \left( \alpha \cdot \sum_{\substack{\text{$I\in \J^{even}$}\\ \text{$QI=n(p+1)$}}} \frac{1}{\pi^{\left| I\right|/2}}\right) - \sum_{I\in \J^{even}} \left( \frac{1}{\pi^{\left| I\right|/2}} \sum_{\substack{\text{$J\in qA^{QI}$}\\ \text{$1\leq \wt(J) \leq n(p+1)$}}} \ind_{\alpha}(J)\cdot \tau_{n-\frac{\wt(J)}{p+1}}\right).
\end{align}
Solving recursion \eqref{recursion ab2340} for $\tau_n$ using the same method using \eqref{elementary eq 1} as in the proof of Theorem \ref{main action thm on u}, setting $x_n$ and $y_n$ as follows,
\begin{align*}
 x_n &= \alpha \cdot \sum_{\substack{\text{$I\in \J^{even}$}\\ \text{$QI=n(p+1)$}}} \frac{1}{\pi^{\left| I\right|/2}},\mbox{\ \ \ and} \\
 y_n &= \sum_{\substack{\text{$I\in \J^{even}$}\\ \text{$J\in qA^{QI}$}\\ \text{$\wt(J)= n(p+1)$}}} \frac{\ind_{\alpha}(J)}{\pi^{\left| I\right|/2}}
\end{align*}
we arrive at \eqref{tau formula 1}.
\end{proof}

Using Corollary \ref{witt action on u thm} and the list of $q$-alternating trees of low weights given in the proof of Corollary \ref{witt comp cor 1}, one can easily write out the first few terms in $\alpha.u$ without any reference to trees:
\begin{corollary}\label{witt comp u cor}
Let $q = p$, and let $\alpha\in W(\mathbb{F}_{q^2})^\times\subseteq\Aut(\mathbb{G})$. Write $\beta$ as shorthand for the quotient $\overline{\alpha}/\alpha$. Then $\alpha$ acts on $u$ as follows:
\begin{align}
\nonumber \alpha . u 
  &\equiv 
   \alpha\cdot u
  +\frac{\alpha}{\pi}(1-\beta^{p+1}) u_1^{p+1} \cdot u
  +\frac{\alpha p}{\pi^2}\beta^{p+1}(1-\beta^{p+1}) u_1^{2p+2}\cdot u \mod (u_1^{3p+3},u_1^{p^2}).\end{align}
\end{corollary}

\appendix
\section{Examples of ordered rooted $q$-labelled trees}
\label{Drawn examples...}
In this appendix, we demonstrate the counting involved in Theorem \ref{main action thm} by drawing all ordered rooted $q$-labelled trees, for all $q$, in all weights $\leq 4$. We will follow the usual conventions in combinatorics by drawing the root of a rooted tree at the {\em top}. We mark each vertex $v$ by the label $(H_v,I_v)$ (see Definition \ref{def of q-labelling}), writing the sequence $H_v$ on top, and the sequence $I_v$ on the bottom. The notation $()$ denotes an empty sequence. We also mark each labelled ordered tree with its weight and $(\alpha_0,\alpha_1)$-index. For brevity we write ``index'' as shorthand for ``$(\alpha_0,\alpha_1)$-index.''

\paragraph{\bf Weights 1 and 2} For all prime powers $q$, there is only one ordered rooted $q$-labelled tree of weight $1$, and only one of weight $2$:
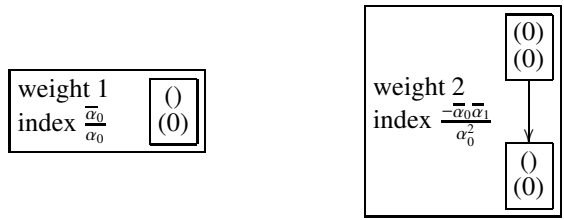
\[
\boxed{\parbox{50pt}{weight $1$\\index $\frac{\overline{\alpha}_0}{\alpha_0}$} \vcenter{\xymatrix{
*+[F]\txt{()\\(0)}}
}}\hspace{60pt}
\boxed{\parbox{50pt}{weight $2$\\index $\frac{-\overline{\alpha}_0\overline{\alpha}_1}{\alpha_0^2}$} \vcenter{\xymatrix{
*+[F]\txt{(0)\\(0)} \ar[d] \\
*+[F]\txt{()\\(0)}}
}}
\]
\paragraph{\bf Weight 3} There is one ordered rooted $q$-labelled tree of weight $3$ which exists for all $q$, and two more that exist only if $q=2$.
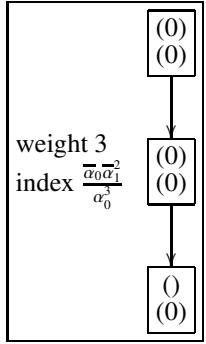
\[
\boxed{\parbox{50pt}{weight $3$\\index $\frac{\overline{\alpha}_0\overline{\alpha}_1^2}{\alpha_0^3}$} \vcenter{\xymatrix{
*+[F]\txt{(0)\\(0)} \ar[d] \\
*+[F]\txt{(0)\\(0)} \ar[d] \\
*+[F]\txt{()\\(0)}}
}}\hspace{60pt}
\boxed{\parbox{50pt}{\textcolor{red}{$q=2$ only}\\weight $3$\\index $\frac{\overline{\alpha}_0^3\alpha_1}{\alpha_0^4}$} \vcenter
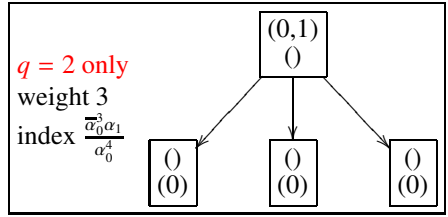
{\xymatrix{
 & *+[F]\txt{(0,1)\\()} \ar[d] \ar[ld] \ar[dr] & \\
*+[F]\txt{()\\(0)} & *+[F]\txt{()\\(0)} & *+[F]\txt{()\\(0)}
}
}}
\]
\[
\boxed{\parbox{50pt}{\textcolor{red}{$q=2$ only}\\weight $3$\\index $\frac{\alpha_1}{\alpha_0}$} \vcenter
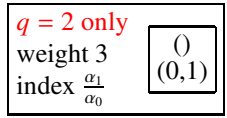
{\xymatrix{
*+[F]\txt{()\\(0,1)}}
}}
\]
\paragraph{\bf Weight 4} There is one ordered rooted $q$-labelled tree of weight $3$ which exists for all $q$, five that exist only if $q=2$, and two that exist only if $q=3$. 

The reader will immediately notice that three of the weight $4$ trees that exist for $q=2$ differ only in their ordering, i.e., they are they horizontal permutations of one another, and they all have the same index, namely $\frac{-\overline{\alpha}_0^3\alpha_1\overline{\alpha}_1}{\alpha_0^5}$. This demonstrates by example that the orderings matter: if we calculate $(\alpha_0 + \alpha_1S).u_1$ in the case $q=2$ using Theorem \ref{main action thm}, the coefficient of $u_1^4$ has {\em three} copies of $\frac{-\overline{\alpha}_0^3\alpha_1\overline{\alpha}_1}{\alpha_0^5}$, not just one, since the sum \eqref{sum 003941} in Theorem \ref{main action thm} is taken over {\em ordered} trees. 
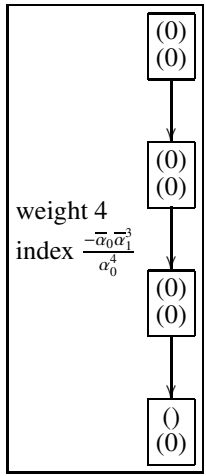
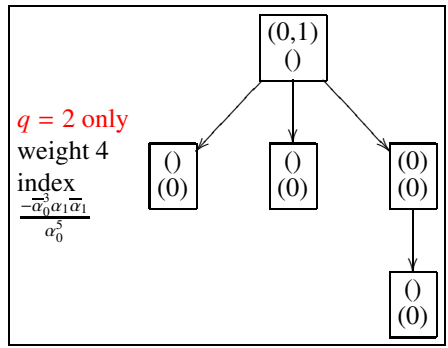
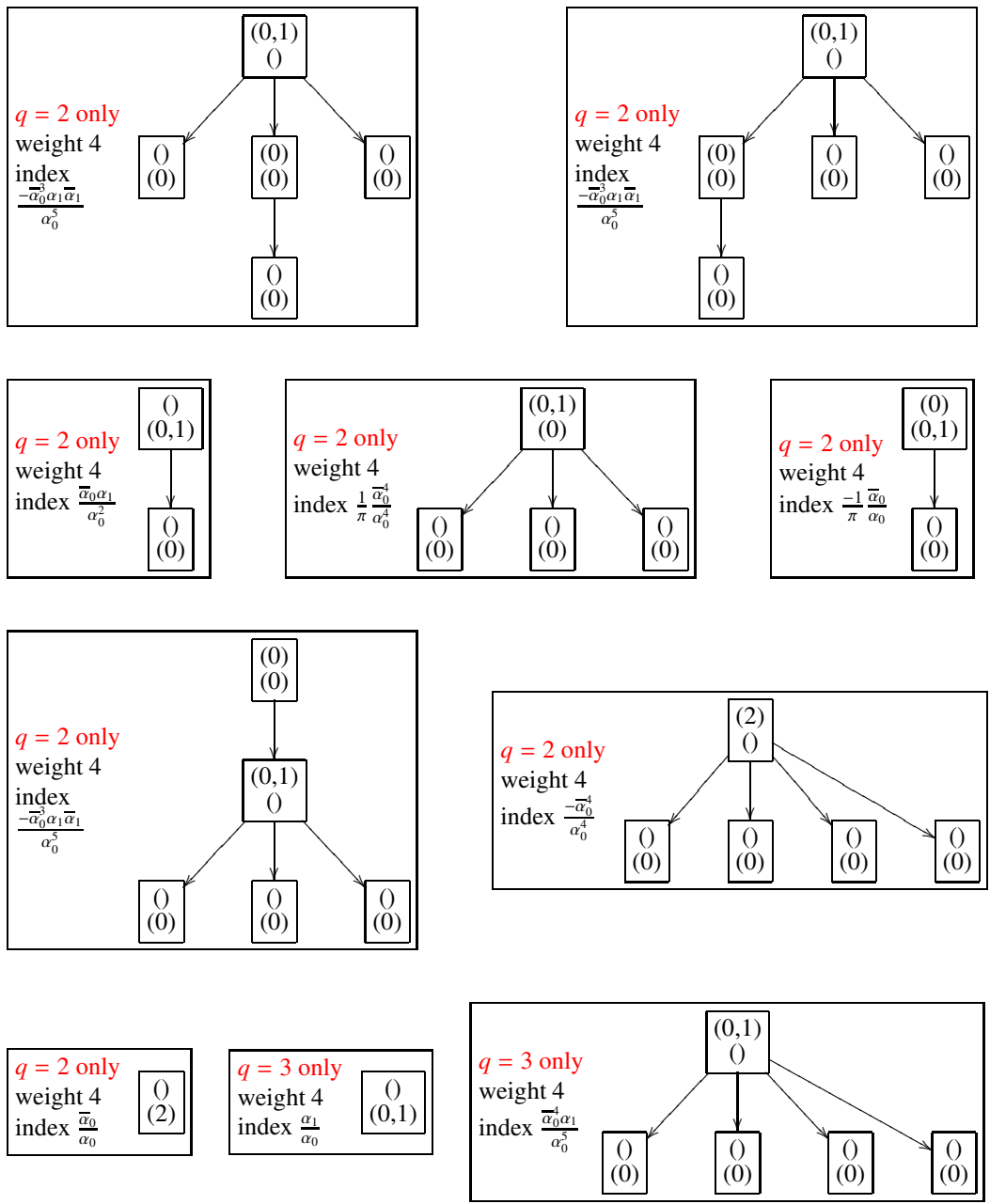
\[
\boxed{\parbox{50pt}{weight $4$\\index $\frac{-\overline{\alpha}_0\overline{\alpha}_1^3}{\alpha_0^4}$} \vcenter{\xymatrix{
*+[F]\txt{(0)\\(0)} \ar[d] \\
*+[F]\txt{(0)\\(0)} \ar[d] \\
*+[F]\txt{(0)\\(0)} \ar[d] \\
*+[F]\txt{()\\(0)}}
}}\hspace{60pt}
\boxed{\parbox{50pt}{\textcolor{red}{$q=2$ only}\\weight $4$\\index $\frac{-\overline{\alpha}_0^3\alpha_1\overline{\alpha}_1}{\alpha_0^5}$} \vcenter{\xymatrix{
 & *+[F]\txt{(0,1)\\()} \ar[d] \ar[ld] \ar[dr] & \\
*+[F]\txt{()\\(0)} & *+[F]\txt{()\\(0)} & *+[F]\txt{(0)\\(0)}\ar[d] \\
 & & *+[F]\txt{()\\(0)}
}
}}
\]
\[
\boxed{\parbox{50pt}{\textcolor{red}{$q=2$ only}\\weight $4$\\index $\frac{-\overline{\alpha}_0^3\alpha_1\overline{\alpha}_1}{\alpha_0^5}$} \vcenter{\xymatrix{
 & *+[F]\txt{(0,1)\\()} \ar[d] \ar[ld] \ar[dr] & \\
*+[F]\txt{()\\(0)} & *+[F]\txt{(0)\\(0)} \ar[d] & *+[F]\txt{()\\(0)}\ \\
 & *+[F]\txt{()\\(0)} &
}
}}
\hspace{60pt}
\boxed{\parbox{50pt}{\textcolor{red}{$q=2$ only}\\weight $4$\\index $\frac{-\overline{\alpha}_0^3\alpha_1\overline{\alpha}_1}{\alpha_0^5}$} \vcenter{\xymatrix{
 & *+[F]\txt{(0,1)\\()} \ar[d] \ar[ld] \ar[dr] & \\
*+[F]\txt{(0)\\(0)} \ar[d] & *+[F]\txt{()\\(0)} & *+[F]\txt{()\\(0)} \\
 *+[F]\txt{()\\(0)} & &
}
}}
\]
\[
\boxed{\parbox{50pt}{\textcolor{red}{$q=2$ only}\\weight $4$\\index $\frac{\overline{\alpha}_0\alpha_1}{\alpha_0^2}$} \vcenter{\xymatrix{
*+[F]\txt{()\\(0,1)} \ar[d] \\
*+[F]\txt{()\\(0)}
}}}
\hspace{30pt}
\boxed{\parbox{50pt}{\textcolor{red}{$q=2$ only}\\weight $4$\\index $\frac{1}{\pi}\frac{\overline{\alpha}_0^4}{\alpha_0^4}$} \vcenter{\xymatrix{
 & *+[F]\txt{(0,1)\\(0)} \ar[d]\ar[ld]\ar[rd] & \\
 *+[F]\txt{()\\(0)} & *+[F]\txt{()\\(0)} & *+[F]\txt{()\\(0)}
}}}
\hspace{30pt}
\boxed{\parbox{50pt}{\textcolor{red}{$q=2$ only}\\weight $4$\\index $\frac{-1}{\pi}\frac{\overline{\alpha}_0}{\alpha_0}$} \vcenter{\xymatrix{
 *+[F]\txt{(0)\\(0,1)} \ar[d] \\
 *+[F]\txt{()\\(0)} 
}}}
\]
\[
\boxed{\parbox{50pt}{\textcolor{red}{$q=2$ only}\\weight $4$\\index $\frac{-\overline{\alpha}_0^3\alpha_1\overline{\alpha}_1}{\alpha_0^5}$} \vcenter{\xymatrix{
 & *+[F]\txt{(0)\\(0)} \ar[d] & \\
 & *+[F]\txt{(0,1)\\()} \ar[ld]\ar[d]\ar[rd] & \\
*+[F]\txt{()\\(0)} & *+[F]\txt{()\\(0)} & *+[F]\txt{()\\(0)}
}}}
\hspace{30pt}
\boxed{\parbox{50pt}{\textcolor{red}{$q=2$ only}\\weight $4$\\index $\frac{-\overline{\alpha}_0^4}{\alpha_0^4}$} \vcenter{\xymatrix{
 & *+[F]\txt{(2)\\()} \ar[d]\ar[ld]\ar[rd]\ar[rrd] & & \\
 *+[F]\txt{()\\(0)} & *+[F]\txt{()\\(0)} & *+[F]\txt{()\\(0)} & *+[F]\txt{()\\(0)}
}}}
\]
\[
\boxed{\parbox{50pt}{\textcolor{red}{$q=2$ only}\\weight $4$\\index $\frac{\overline{\alpha}_0}{\alpha_0}$} \vcenter{\xymatrix{
 *+[F]\txt{()\\(2)} 
}}}
\hspace{15pt}
\boxed{\parbox{50pt}{\textcolor{red}{$q=3$ only}\\weight $4$\\index $\frac{\alpha_1}{\alpha_0}$} \vcenter{\xymatrix{
 *+[F]\txt{()\\(0,1)} 
}}}
\hspace{15pt}
\boxed{\parbox{50pt}{\textcolor{red}{$q=3$ only}\\weight $4$\\index $\frac{\overline{\alpha}_0^4\alpha_1}{\alpha_0^5}$} \vcenter{\xymatrix{
 & *+[F]\txt{(0,1)\\()} \ar[ld]\ar[d]\ar[rd]\ar[rrd] & & \\
 *+[F]\txt{()\\(0)} & *+[F]\txt{()\\(0)} & *+[F]\txt{()\\(0)} & *+[F]\txt{()\\(0)} 
}}}
\]
One can of course go on to higher weights, but we stop here, as weight $4$ is enough to see the first case in which the orderings on the trees matter.

We return now to the example from \eqref{p=2 example}. In that case, $q=2$ and $\alpha_0+ \alpha_1S = 1 + \zeta_3 p$, i.e., $\alpha_0 =1 + \zeta_3 p$ and $\alpha_1 = 0$, where $\zeta_3$ is a primitive cube of unity in $W(\mathbb{F}_4)$. It is elementary to calculate that $\overline{\alpha}_0/\alpha_0$ is then equal to $-1$.
Looking over the $2$-labelled trees drawn above, we see that the only weight $1$ tree has index $\overline{\alpha}_0/\alpha_0$, while all the weight $2$ and weight $3$ trees have index divisible by $\alpha_1$. This, together with Theorem \ref{lettered thm A}, tells us that $(1+\zeta_3p).u_1\equiv -u_1\mod u_1^4$. Among the weight $4$ trees, there are four whose indices are not divisible by $\alpha_1$. Summing their indices, we get 
\begin{align*} \frac{1}{p}\frac{\overline{\alpha}_0^4}{\alpha_0^4} - \frac{1}{p}\frac{\overline{\alpha}_0}{\alpha_0} - \frac{\overline{\alpha}_0^4}{\alpha_0^4} + \frac{\overline{\alpha}_0}{\alpha_0}
 &= -1, \mbox{\ \ \ i.e.,}\\
(1+\zeta_3p).u_1\equiv -u_1 - u_1^4\mod u_1^5.
\end{align*}
This confirms the first terms in the power series \eqref{p=2 example} and shows how they arise from $2$-labelled trees.


\def\cprime{$'$} \def\cprime{$'$} \def\cprime{$'$} \def\cprime{$'$}

\end{document}